\documentclass[11pt,leqno]{amsart}
\usepackage{verbatim,amssymb,amsfonts,latexsym,amsmath,amsthm,tikz,bbm,nicefrac,xspace,graphicx,mathtools,multicol}
\usetikzlibrary{hobby}
\usepackage[T1]{fontenc}

\newtheorem{theorem}{Theorem}[section]
\newtheorem{lemma}[theorem]{Lemma}

\newtheorem{question}[theorem]{Question}

\newtheorem*{introtheorem}{Theorem~\ref{thm:noseq}}
\newtheorem*{theorem*}{Theorem}
\newtheorem*{corollary*}{Corollary}
\newtheorem*{claim1}{Claim 1}
\newtheorem*{claim2}{Claim 2}
\newtheorem*{claim3}{Claim 3}
\newtheorem*{claim4}{Claim 4}

\newtheorem*{claim*}{Revised Claim}

\newtheorem*{sub-claim}{sub-claim}
\theoremstyle{definition}

\newtheorem{definition}[theorem]{Definition}

\theoremstyle{remark}

\newtheorem*{definition*}{Definition}

\newcommand{\Q}{\mathbb{Q}}
\newcommand{\R}{\mathbb{R}}
\newcommand{\N}{\mathbb{N}}

\newcommand{\A}{\mathcal{A}}
\newcommand{\B}{\mathcal{B}}
\newcommand{\C}{\mathcal{C}}

\newcommand{\F}{\mathcal{F}}

\newcommand{\U}{\mathcal{U}}

\newcommand{\explicitSet}[1]{\left\lbrace #1 \right\rbrace}
\newcommand{\brackets}[1]{\left\langle #1 \right\rangle}
\newcommand{\set}[2]{\explicitSet{#1 \colon #2}}
\newcommand{\seq}[2]{\brackets{#1 \colon #2}}
\newcommand{\<}{\langle}
\renewcommand{\>}{\rangle}
\renewcommand{\a}{\alpha}
\renewcommand{\b}{\beta}

\newcommand{\dlt}{\delta}
\newcommand{\e}{\varepsilon}
\newcommand{\z}{\zeta}
\newcommand{\kp}{\kappa}
\newcommand{\s}{\sigma}

\newcommand{\w}{\omega}
\newcommand{\0}{\emptyset}
\newcommand{\sub}{\subseteq}
\newcommand{\rest}{\!\restriction\!}
\newcommand{\cat}{\!\,^{{}_{{}^\frown\!}}}

\newcommand{\closure}[1]{\overline{#1}}

\newcommand{\boundary}{\partial}

\newcommand{\cf}{\mathrm{cf}}
\newcommand{\card}[1]{\left\lvert #1 \right\rvert}
\newcommand{\PP}{\mathbb{P}}

\newcommand{\forces}{\Vdash}

\newcommand{\1}{\mathbbm{0}}
\renewcommand{\1}{\mathbbm{1}}
\newcommand{\val}{\mathrm{val}}
\newcommand{\stem}[1]{\mathrm{stem}(#1)}
\newcommand{\cohen}[1]{\mathrm{Fn}(#1,2)}

\newcommand{\pnmf}{\mathcal{P}(\N)/\mathrm{fin}}

\newcommand{\DD}{\mathbb{D}}
\newcommand{\continuum}{\mathfrak{c}}

\newcommand{\dom}{\mathfrak d}
\newcommand{\bdd}{\mathfrak b}
\renewcommand{\split}{\mathfrak s}
\newcommand{\bld}[1]{\ensuremath{\mathbf {#1}}}
\newcommand{\scr}[1]{\ensuremath{\mathcal {#1}}}
\newcommand{\add}[1]{\ensuremath{\bld{add}(\scr{#1})}}
\newcommand{\non}[1]{\ensuremath{\bld{non}(\scr{#1})}}
\newcommand{\cov}[1]{\ensuremath{\bld{cov}(\scr{#1})}}
\newcommand{\cof}[1]{\ensuremath{\bld{cof}(\scr{#1})}}
\newcommand{\ch}{\ensuremath{\mathsf{CH}}\xspace}
\newcommand{\zfc}{\ensuremath{\mathsf{ZFC}}\xspace}

\newcommand{\ma}{\ensuremath{\mathsf{MA}}\xspace}

\newcommand{\noseq}{\mathfrak{z}}
\newcommand{\sr}{\mathfrak{s}(\R)}
\newcommand{\gen}[1]{\<\!\!\<#1\>\!\!\>}

\begin{document}

\title{Small cardinals and small Efimov spaces}
\author{Will Brian}
\address {
Will Brian\\
Department of Mathematics and Statistics\\
University of North Carolina at Charlotte\\
9201 University City Blvd.\\
Charlotte, NC 28223}
\email{wbrian.math@gmail.com}
\author{Alan Dow}
\address {
Alan Dow\\
Department of Mathematics and Statistics\\
University of North Carolina at Charlotte\\
9201 University City Blvd.\\
Charlotte, NC 28223}
\email{adow@uncc.edu}

\begin{abstract}
We introduce and analyze a new cardinal characteristic of the continuum, the \emph{splitting number of the reals}, denoted $\sr$. This number is connected to Efimov's problem, which asks whether every infinite compact Hausdorff space must contain either a non-trivial convergent sequence, or else a copy of $\b\N$. 
\end{abstract}

\maketitle

\section{Introduction}

This paper is about a new cardinal characteristic of the continuum, the \emph{splitting number of the reals}, denoted $\sr$.

\begin{definition}\label{def:sr}
If $U$ and $A$ are infinite sets, we say that $U$ \emph{splits} $A$ provided that both $A \cap U$ and $A \setminus U$ are infinite. 
The cardinal number $\sr$ is defined as the smallest cardinality of a collection $\U$ of open subsets of $\R$ such that every infinite $A \sub \R$ is split by some $U \in \U$.
\end{definition}

\noindent In this definition, $\R$ is assumed to have its usual topology.
The number $\sr$ is a topological variant of the splitting number $\split$, and is a cardinal characteristic of the continuum in the sense of \cite{Blass}. 

Most of this paper is devoted to understanding the place of $\sr$ among the classical cardinal characteristics of the continuum.
Our main achievement along these lines is to determine completely the place of $\sr$ in Cicho\'n's diagram. More precisely, for every cardinal $\kp$ appearing in Cicho\'n's diagram, we prove either that $\kp$ is a (consistently strict) lower bound for $\sr$, or that $\kp$ is a (consistently strict) upper bound for $\sr$, or else that each of $\kp < \sr$ and $\sr < \kp$ is consistent. 

\vspace{2mm}
\begin{center}
\begin{tikzpicture}[xscale=1.14,yscale=1.44]

\path[fill=green!20, draw=white] (-.5,-.5) -- (7,-.5) -- (7,.5) -- (5,.5) -- (5,1.5) -- (-.5,1.5) -- (-.5,-.5);
\path[fill=yellow!20, draw=white] (-.5,1.5) -- (-.5,2.5) -- (7,2.5) -- (7,1.5) -- (10.5,1.5) -- (10.5,-.5) -- (7,-.5) -- (7,.5) -- (5,.5) -- (5,1.5) -- (-.5,1.5) -- (-.5,2.5);
\path[fill=red!20, draw=white] (7,2.5) -- (10.5,2.5) -- (10.5,1) -- (7.5,1) -- (7.5,1) -- (7,1) -- (7,2.5);

\draw (0,0) -- (4,0);
\draw[dashed] (4.1,0) -- (6,0);
\draw (6,0) -- (8,0);
\draw (2,2) -- (4,2);
\draw[dashed] (4.1,2) -- (6,2);
\draw (6,2) -- (10,2);
\draw (2,0) -- (2,2);
\draw[dashed] (4,-0.05) -- (4,1);
\draw (4,1) -- (4,2);
\draw (4,1) -- (6,1);
\draw (6,0) -- (6,1);
\draw[dashed] (6,.95) -- (6,2);
\draw (8,0) -- (8,2);

\draw [fill=green!20,green!20] (0,0) circle (5pt);  \node at (0,0) {\footnotesize $\aleph_1$};
\draw [fill=green!20,green!20] (2,0) ellipse (17pt and 5pt);  \node at (2,0) {\footnotesize $\add N$};
\draw [fill=green!20,green!20] (4,0) ellipse (18pt and 5pt);  \node at (4,0) {\footnotesize $\add M$};
\draw [fill=green!20,green!20] (6,0) ellipse (17pt and 5pt);  \node at (6,0) {\footnotesize $\cov M$};
\draw [fill=yellow!20,yellow!20] (8,0) ellipse (17pt and 5pt);  \node at (8,0) {\footnotesize $\non N$};
\draw [fill=green!20,green!20] (4,1) circle (5pt);  \node at (4,1) {\footnotesize $\bdd$};
\draw [fill=yellow!20,yellow!20] (6,1) circle (5pt);  \node at (6,1) {\footnotesize $\dom$};
\draw [fill=yellow!20,yellow!20] (2,2) ellipse (16pt and 5pt);  \node at (2,2) {\footnotesize $\cov N$};
\draw [fill=yellow!20,yellow!20] (4,2) ellipse (18pt and 5pt);  \node at (4,2) {\footnotesize $\non M$};
\draw [fill=yellow!20,yellow!20] (6,2) ellipse (17pt and 5pt);  \node at (6,2) {\footnotesize $\cof M$};
\draw [fill=red!20,red!20] (8,2) ellipse (16pt and 5pt);  \node at (8,2) {\footnotesize $\cof N$};
\draw [fill=red!20,red!20] (10,2) circle (4pt);  \node at (10,2) {\footnotesize $\continuum$};

\node at (9.25,1.67) {\footnotesize upper};
\node at (9.25,1.45) {\footnotesize bounds};
\node at (9.4,.6) {\footnotesize incomparable};
\node at (9.4,.35) {\footnotesize cardinals};
\node at (.75,.95) {\footnotesize lower};
\node at (.75,.7) {\footnotesize bounds};

\end{tikzpicture}
\end{center}
\vspace{2mm}



To explain our motivation for investigating the cardinal number $\sr$, we begin with a longstanding open problem of set-theoretic topology:

\vspace{2mm}

\noindent \textbf{Efimov's problem \cite{efimov,hart}:} Does every infinite compact Hausdorff space contain either a non-trivial convergent sequence, or else a copy of $\b\N$?

\vspace{2mm}

\noindent An \emph{Efimov space} is defined to be an infinite compact Hausdorff space containing neither a non-trivial convergent sequence nor a copy of $\b\N$, should such a space exist. 

Since the 1970's, it has been known that a negative solution to Efimov's problem is consistent with \zfc. In other words, it is consistent that Efimov spaces exist. This result is due to Fedor\v{c}uk, who published three separate papers giving three separate constructions of Efimov spaces, each time using a different set-theoretic axiom to facilitate his construction \cite{fedorcuk1, fedorcuk2, fedorcuk3}. 
It is unknown whether a positive solution to Efimov's problem is also consistent, and in this sense the problem remains open.

The cardinal $\sr$ is related to the question of how ``small'' an Efimov space can be, where we measure the smallness of a space by its weight. Recall that the \emph{weight} of a topological space is the smallest size of a basis for that space.
From now on, a \emph{space} will always mean an infinite Hausdorff topological space.
A space is called \emph{non-sequential} if it is not discrete and yet contains no non-trivial convergent sequences.

\begin{definition}
The cardinal number $\noseq$ is defined to be the smallest weight of a compact non-sequential space.
\end{definition}

\noindent The notation $\noseq$ is from the Polish \emph{zbie$\dot{z}$no\'s\'c} meaning \emph{convergence}, and was suggested by Damian Sobota in \cite{Sobota}. Let us record a few observations:

\begin{itemize}
\item[$\circ$] Booth proved in \cite{Booth} that every compact space of weight $< \! \split$ is sequentially compact, and thus contains non-trivial convergent sequences. Hence $\split \leq \noseq$.
\item[$\circ$] Koppelberg proved in \cite{Koppelberg} that every compact space with weight ${<\!\cov M}$ contains a non-trivial convergent sequence. (Her proof is phrased in terms of Boolean algebras; see \cite[Section 2]{Geschke} for a topological translation.) Hence $\cov M \leq \noseq$.
\item[$\circ$] The space $\b\N$, the Stone-\v{C}ech compactification of the countable discrete space $\N$, contains no non-trivial converging sequences, and it follows that $\noseq \leq \continuum = \mathrm{weight}(\b\N)$.
\item[$\circ$] If $\noseq < \continuum$, then there is an Efimov space. This is because any space containing a copy of $\b\N$ must have weight $\geq\!\continuum$. If there is an Efimov space, then $\noseq = \min \set{\mathrm{weight}(X)}{X \text{ is an Efimov space}}$.
\end{itemize}

Thus the number $\noseq$ could be considered a cardinal characteristic of the continuum, one closely tied to Efimov's problem. Unfortunately, $\noseq$ seems difficult to analyze directly. In this paper, we analyze it indirectly by examining the cardinal $\sr$ instead. The cardinals $\sr$ and $\noseq$ are related by the following theorem, whose proof is found in Section~\ref{sec:topology} below.

\begin{introtheorem}
Suppose $\kp$ is a cardinal such that $\sr \leq \kp = \mathrm{cof}(\kp^{\aleph_0},\sub)$. Then $\noseq \,\leq\, \kp$.
\end{introtheorem}

\noindent Recall that $\mathrm{cof}(\kp^{\aleph_0},\sub)$ denotes the smallest possible size of a collection $\C$ of countable subsets of $\kp$ such that every countable subset of $\kp$ is contained in some member of $\C$.
If $\kp$ has uncountable cofinality, and if the Covering Lemma holds with respect to some inner model $\mathrm K$ satisfying the $\mathsf{GCH}$, then $\kp = \mathrm{cof}(\kappa^{\aleph_0},\sub)$; see \cite[Section 4]{JMPS} for a proof. Conversely, if $\cf(\kp) > \w$ then the inequality $\kp < \mathrm{cof}(\kp^{\aleph_0},\sub)$ implies the Covering Lemma fails over any such inner model, an assertion of significant large cardinal strength.
(It implies at least that there is an inner model containing a measurable cardinal by results in \cite{DoddJensen}; on the other hand, Gitik obtained in \cite{Gitik} a model in which $\mathrm{cof}(\aleph_{\w+1}^{\aleph_0},\sub) > \aleph_{\w+1}$ using a measurable cardinal $\kp$ of Mitchell order $o(\kp) = \kp^{++}$.) 

Furthermore, the inequality $\kp < \mathrm{cof}(\kp^{\aleph_0},\sub)$ implies $\kp \geq \aleph_\w$. (Again, a proof can be found in \cite[Section 4]{JMPS}.) Thus if $\sr < \aleph_\w$, then we may take $\kp = \sr$ in the above theorem. 
\begin{corollary*}
If $\sr < \aleph_\w$ then $\noseq \leq \sr$.
\end{corollary*}

The present paper is organized as follows. In Section~\ref{sec:topology} we prove some basic facts about $\sr$ and prove the theorem concerning $\sr$ and $\noseq$ discussed above.
After this, we prove several theorems comparing $\sr$ with other cardinal characteristics of the continuum. Specifically, in Sections~\ref{sec:lower} and \ref{sec:upper} we prove the following bounds on $\sr$ in terms of more familiar cardinals:

\begin{multicols}{2}
\begin{itemize}
\item $\split \leq \sr$ $\vphantom{ \sum_1}$
\item $\bdd \leq \sr$ $\vphantom{ \sum^2}$
\item $\cov M \leq \sr$ $\vphantom{ \sum_1}$
\item $\sr \leq \max \{ \bdd, \non N \}$ $\vphantom{ \sum^2}$
\end{itemize}
\end{multicols}

\noindent (For the definitions of these other cardinal characteristics of the continuum, we refer the reader to \cite{Blass} or \cite{vanDouwen}.) We show that each of these four bounds is consistently strict.
In Section~\ref{sec:cichon}, we also show via forcing the relative consistency of two more inequalities:

\begin{multicols}{2}
\begin{itemize}
\item $\bdd = \sr < \non N$ 
\item $\dom = \cof M < \sr$ 
\end{itemize}
\end{multicols}

\noindent These six results, together with established facts about the cardinals in Cicho\'n's diagram, suffice to pinpoint the location of $\sr$ in Cicho\'n's diagram as described above. This is discussed further in Section~\ref{sec:cichon}. We have also included a short Section~\ref{sec:questions} reviewing the status of Efimov's problem, and stating some open questions related to $\sr$ and $\noseq$.

\section{More on $\sr$ and $\noseq$}\label{sec:topology}

We begin this section by showing that the value of $\sr$ remains unchanged when $\R$ is replaced by any other uncountable Polish space in Definition~\ref{def:sr}, or when the definition is altered by requiring only certain kinds of infinite sets to be split by every $U \in \U$.

\begin{definition}\label{def:sr2}
If $X$ is a topological space, then the number $\split(X)$ is defined as the smallest cardinality of a collection $\U$ of open subsets of $X$ such that every infinite $A \sub X$ is split by some $U \in \U$.
\end{definition}

\begin{lemma}\label{lem:subspace}
If $Y \sub X$, then $\split(Y) \leq \split(X)$.
\end{lemma}
\begin{proof}
If $\U$ is a family of open subsets of $X$ such that every infinite $A \sub X$ is split by some $U \in \U$, then every infinite $A \sub Y$ is split by some $U \in \U$.
\end{proof}

\begin{lemma}\label{lem:0dim}
Suppose $X$ is an uncountable, zero-dimensional, Borel subspace of a Polish space. Then $\split(X) = \split(2^\w)$.
\end{lemma}
\begin{proof}
If $X$ is as in the statement of the lemma, then $X$ contains a copy of the Cantor space \cite[Theorem 6.2]{Kechris} and $X$ embeds topologically into the Cantor space \cite[Theorem 7.3]{Kechris}. By the previous lemma, this implies that $\split(2^\w) \leq \split(X) \leq \split(2^\w)$, so $\split(X) = \split(2^\w)$.
\end{proof}

\begin{lemma}\label{lem:split}
$\split(2^\w)$ is uncountable.
\end{lemma}
\begin{proof}
It follows immediately from the definitions that if $\N$ has the discrete topology, then $\split(\N)$ is equal to the splitting number $\split$ (which is uncountable). As $\split \leq \split(2^\w)$ by Lemma~\ref{lem:subspace}, $\sr$ is uncountable.
\end{proof}

\begin{theorem}\label{thm:Polish}
If $X,Y$ are uncountable Polish spaces, then $\split(X) = \split(Y)$.
\end{theorem}
\begin{proof}
Let $H$ denote the Hilbert cube $[0,1]^{\aleph_0}$. To prove the theorem, it suffices to show that $\split(H) \leq \split(2^\w)$. This is because, if $X$ is any uncountable Polish space, then $X$ contains a copy of the Cantor space \cite[Theorem 6.2]{Kechris} and $X$ embeds into the Hilbert cube \cite[Theorem 4.14]{Kechris}. By Lemma~\ref{lem:subspace}, it follows that $\split(2^\w) \leq \split(X) \leq \split(H)$, so if $\split(H) \leq \split(2^\w)$ then $\split(H) = \split(2^\w) = \split(X)$ for every uncountable Polish space $X$.

To see that $\split(H) \leq \split(2^\w)$, we use a slight variation of a result of Hausdorff \cite{Hausdorff}, which states that the Baire space $\w^\w$ can be written as an increasing union $\bigcup_{\a < \w_1} X_\a$, where each $X_\a$ is a $G_{\dlt}$ subspace of $\w^\w$.

This result is a relatively straightforward consequence of the existence of Hausdorff gaps. Recall that a Hausdorff gap is a sequence $\seq{(f_\a,g_\a)}{\a < \w_1}$ of pairs of functions $\w \to \w$ such that 
\begin{itemize}
\item $f_\a <^* f_\b <^* g_\b <^* g_\a$ for every $\a < \b < \w_1$ (where, as usual, $f <^* g$ means that $f(n) < g(n)$ for all but finitely many $n \in \w$).
\item there is no $h: \w \to \w$ such that $f_\a <^* h <^* g_\a$ for all $\a < \w_1$.
\end{itemize}
Taking $X_\a = \set{f \in \w^\w}{f_\a <^* f <^* g_\a}$ for each $\a < \w_1$, one may check that each $X_\a$ is $G_\dlt$ and that $\w^\w = \bigcup_{\a < \w_1}X_\a$. 

Recall that $\w^\w$ is homeomorphic to $[0,1] \setminus \Q$, so we may write $[0,1] \setminus \Q = \bigcup_{\a < \w_1}X_\a$,  and by setting $Y_n = X_n \cup \{$the $n^{\mathrm{th}}$ rational number in $[0,1]\}$ and $Y_\a = X_\a$ for $\a \geq \w$, we get  $[0,1] = \bigcup_{\a < \w_1}Y_\a$, where each $Y_\a$ is a zero-dimensional, $G_{\dlt}$ subspace of $[0,1]$.
Furthermore, each $Y_\a$ is uncountable (because each $X_\a$ is, as one may easily check).
Taking $Z_\a = Y_\a^\w$ for every $\a < \w_1$, we arrive at the desired modification of Hausdorff's result: the Hilbert cube $H$ is an increasing union $\bigcup_{\a < \w_1}Z_\a$ of uncountable, Borel, zero-dimensional subspaces. 

For each $\a < \w_1$, let $\U_\a$ be a family of at most $\split(2^\w)$ open subsets of $H$ such that every infinite $A \sub Z_\a$ is split by some $U \in \U_\a$. (Some such $\U_\a$ exists by Lemma~\ref{lem:0dim}.) Let $\U = \bigcup_{\a < \w_1}\U_\a$. 

Suppose $A \sub H$ is infinite. Then there is some $\a < \w_1$ such that $A \cap Z_\a$ is infinite. Thus there is some $U \in \U_\a$ such that both $(A \cap Z_\a) \cap U$ and $(A \cap Z_\a)\setminus U$ are infinite. But then both $A \cap U$ and $A \setminus U$ are infinite as well, so $U$ splits $A$.

This shows that $\split(H) \leq |\U| = \card{\bigcup_{\a < \w_1}U_\a} \leq \aleph_1 \cdot \split(2^\w)$. Lemma~\ref{lem:split} implies that $\aleph_1 \cdot \split(2^\w) = \split(2^\w)$, so $\split(H) \leq \split(2^\w)$, as desired.
\end{proof}

\begin{theorem}
Let $X$ be a Polish space. Let $[X]^\w$ denote the set of all countably infinite subsets of $X$, and let $[X]^{con}$ denote the set of all non-trivial convergent sequences in $X$, considered as sets rather than sequences. If we define
\begin{itemize}
\item $\split^\w(X)$ to be the smallest cardinality of a collection $\U$ of open subsets of $X$ such that every $A \in [X]^\w$ is split by some $U \in \U$, and
\item $\split^{con}(X)$ to be the smallest cardinality of a collection $\U$ of open subsets of $X$ such that every $A \in [X]^{con}$ is split by some $U \in \U$
\end{itemize}
then $\split^\w(X) = \split^{con}(X) = \split(X) = \sr$.
\end{theorem}
\begin{proof}
Clearly $[X]^{con} \sub [X]^\w$ and every $A \in [X]^\w$ is infinite. It follows that $\split^{con}(X) \leq \split^\w(X) \leq \split(X)$.

As in the proof of Lemma~\ref{lem:subspace}, if $\U$ is any family that splits every $A \in [X]^{con}$, then $\U$ also splits every $A \in [Y]^{con}$ for any subspace $Y$ of $X$. As $2^\w$ embeds in $X$ \cite[Theorem 6.2]{Kechris}, it follows that $\split^{con}(2^\w) \leq \split^{con}(X)$. Thus
$$\split^{con}(2^\w) \leq \split^{con}(X) \leq \split^\w(X) \leq \split(X) = \split(2^\w)$$
(where the final equality follows from the previous theorem). Thus, to finish the proof of the theorem, it suffices to show that $\split(2^\w) \leq \split^{con}(2^\w)$.

Suppose $\U$ is a collection of open subsets of $2^\w$ with $|\U| \leq \split^{con}(2^\w)$ such that every $A \in [2^\w]^{con}$ is split by some $U \in \U$. If $B$ is any infinite subset of $2^\w$, then (because $2^\w$ is a compact metrizable space) there is some $A \sub B$ with $A \in [2^\w]^{con}$, and hence some $U \in \U$ that splits $A$. But then $U$ splits $B$ as well. Thus every infinite subset of $2^\w$ is split by some $U \in \U$, and it follows that $\split(2^\w) \leq \split^{con}(2^\w)$ as desired.
\end{proof}

We now move on to the proof of the theorem mentioned in the introduction connecting $\sr$ with $\noseq$. The proof bears some resemblance to Fedor\v{c}uk's construction of an Efimov space in \cite{fedorcuk3}, which uses the hypotheses $2^{\aleph_0} = 2^{\aleph_1}$ and $\split = \aleph_1$. It bears an even stronger resemblance to van Douwen and Fleissner's refinement of Fedor\v{c}uk's construction in \cite[Section 2.5]{vanDouwen&Fleissner}, where they construct an Efimov space of weight $\aleph_1$ from the hypotheses $2^{\aleph_0} = 2^{\aleph_1}$ and $\sr = \aleph_1$ (although, of course, the latter hypothesis is not phrased in this way). In some sense, the proof presented here simply optimizes their strategy, and isolates the critical cardinal invariant for the construction.

We also point out that a similar theorem was obtained by Damian Sobota in \cite[Section 8]{Sobota2}. He showed that if there is a cardinal $\kp$ such that $\cof N \leq \kp = \mathrm{cof}(\kp^{\aleph_0},\sub) < \continuum$, then there is an Efimov space of weight $\kp$. Equivalently: if $\cof N \leq \kp = \mathrm{cof}(\kp^{\aleph_0},\sub) < \continuum$ then $\noseq \leq \kp$.
As $\sr \leq \max\{\bdd,\non N\} \leq \cof N$ (with both inequalities being consistently strict), Theorem~\ref{thm:noseq} can be viewed as a tightening of Sobota's result.

\begin{lemma}\label{lem:splitlb}
$\split \leq \sr$.
\end{lemma}
\begin{proof}
The proof of Lemma~\ref{lem:split} shows that $\split \leq \split(2^\w)$, and $\split(2^\w) = \sr$ by Theorem~\ref{thm:Polish}. (Alternatively, one may note that the proof of Lemma~\ref{lem:split} applies to $\sr$ just as well as $\split(2^\w)$.)
\end{proof}

A subset $U$ of a topological space is called \emph{regular open} if $U = \mathrm{int}(\closure{U})$. 
The regular open subsets of a topological space $X$ form a complete Boolean algebra, called the \emph{regular open algebra of $X$} and denoted $\mathsf{ro}(X)$. If $X$ is a Stone space (i.e., compact, Hausdorff, and zero-dimensional) then $\mathsf{ro}(X)$ is the Boolean completion of $\mathsf{clop}(X)$, the Boolean algebra consisting of all clopen subsets of $X$. 

In the following two proofs, we will be looking at $\mathsf{ro}(2^\w)$, its subalgebras, and their Stone spaces. If $A \sub \mathsf{ro}(2^\w)$, then $\gen{A}$ denotes the subalgebra of $\mathsf{ro}(2^\w)$ generated by $A$. For any subalgebra $\B$ of $\mathsf{ro}(2^\w)$, we denote its Stone space by $\mathsf{st}(\B)$. Recall that $\mathsf{st}(\B)$ consists of all ultrafilters on $\B$. Thus whenever $\A$ is a subalgebra of $\B$: if $x$ is a point of $\mathsf{st}(\B)$ then $x \cap \A$ is a point of $\mathsf{st}(\A)$ and, conversely, if $x \in \mathsf{st}(\A)$ then there is some $y \in \mathsf{st}(\B)$ with $y \cap \mathsf{st}(\A) = x$.

\begin{lemma}\label{lem:noseq}
Suppose $\B$ is a countable Boolean algebra such that 
$$\mathsf{clop}(2^\w) \sub \B \sub \mathsf{ro}(2^\w).$$
There is a subset $\U \sub \mathsf{ro}(2^\w)$ with $\card{\U} = \sr$ having the following property:
\begin{itemize}
\item[$(*)$] Let $\seq{x_n}{n \in \N}$ be a sequence of points in $\mathsf{st}(\B)$, and for each $n$ let $y_n$ be a point of $\mathsf{st}(\gen{\B \cup \U})$ such that $y_n \cap \B = x_n$. Then the sequence $\seq{y_n}{n \in \N}$ does not converge in $\mathsf{st}(\gen{ \B \cup \U })$.
\end{itemize} 
\end{lemma}
\begin{proof}
Let $\B$ be a countable Boolean algebra with 
$\mathsf{clop}(2^\w) \sub \B \sub \mathsf{ro}(2^\w)$. $\B$ is atomless because $\B \supseteq \mathsf{clop}(2^\w)$ and $\mathsf{clop}(2^\w)$ is atomless and dense in $\mathsf{ro}(2^\w)$. As $\B$ is a countable, atomless Boolean algebra, it is isomorphic to $\mathsf{clop}(2^\w)$, and $\mathsf{st}(\B)$ is homeomorphic to $2^\w$.
Hence there is a family $\U_0$ of open subsets of $\mathsf{st}(\B)$ with $|\U_0| = \sr$ such that every member of $[\mathsf{st}(\B)]^{con}$ is split by some $U \in \U_0$. 

Observe that no $X \in [\mathsf{st}(\B)]^{con}$ is split by a clopen set. (If $x$ is the limit point of $X$, then $X \cap C$ is either finite or co-finite, depending on whether $x \notin C$ or $x \in C$ respectively.) Thus, by removing any clopen sets from $\U_0$ if necessary, we may (and do) assume that no $U \in \U_0$ is clopen. 

For each $b \in \B$ let $[b] = \set{x \in \mathsf{st}(\B)}{b \in x}$, and recall that the sets of this form constitute the canonical basis of clopen sets for $\mathsf{st}(\B)$. For each $U \in \U_0$, let us fix a sequence $\seq{b^U_n}{n \in \N}$ of nonempty, pairwise disjoint members of $\B$ such that $U = \bigcup_{n \in \N}[b^n_U]$. This is possible because $\B$ is countable and because $U$ is open, but not clopen, in $\mathsf{st}(\B)$.

Let $\F$ be a family of infinite subsets of $\N$ such that $\card{\F} = \split$ and every infinite subset of $\N$ is split by some $A \in \F$.
For each $A \in \F$ and $U \in \U_0$, define $R(U,A) = \mathrm{int}\!\left(\closure{\bigcup_{n \in A}b^n_U}\right)$. Each $b^n_U$ is an open subset of $2^\w$; in this definition, the interior and the closure are taken in $2^\w$, not in $\mathsf{st}(\B)$.
Let
$$\U = \set{R(U,A)}{U \in \U_0 \text{ and } A \in \F}.$$
We claim that this $\U$ is as required. 

One may easily check that $\mathrm{int}\!\left(\closure{\mathrm{int}(\closure{V})}\right) = \mathrm{int}(\closure{V})$ for any open set $V$ (in any space).
Hence each $R(U,A)$ is a regular open subset of $2^\w$.
Also, $|\U| = |\U_0| \cdot |\F| = \sr \cdot \split = \sr$ (where the final equality holds because $\split \leq \sr$ by Lemma~\ref{lem:splitlb}). It remains to check property $(*)$. 

Let $\seq{x_n}{n \in \N}$ be a sequence of points in $\mathsf{st}(\B)$, and for each $n$ let $y_n$ be a point of $\mathsf{st}(\gen{\B \cup \U})$ such that $y_n \cap \B = x_n$. Let $X = \set{x_n}{n \in \N}$ and $Y = \set{y_n}{n \in \N}$, and fix some $U \in \U_0$ that splits $X$. Note that the limit point of $X$ is not in $U$ (as then $U$ would not split $X$). Let 
$$B = \set{m \in \N}{X \cap [b^m_U] \neq \0}.$$
Any particular $[b^m_U]$ contains only finitely many points of $X$, because it is closed and $X$ converges to a point outside $[b^m_U]$. This implies $B$ is infinite. 
Thus there is some $A \in \F$ such that $A$ splits $B$. 
$R(U,A) \in \U$, and we claim that $[R(U,A)]$ splits $Y$ in $\mathsf{st}(\gen{\B \cup \U})$.

To see this, first notice that if $m \in A$ then $b^m_U \sub R(U,A)$, and if $m \notin A$ then $b^m_U \cap A = \0$. (This follows easily from the definition of $R(U,A)$ and the fact that the $b^m_U$ are open and pairwise disjoint.)

Because $A$ splits $B$, there are infinitely many $m \in A$ such that $[b^m_U] \cap X \neq \0$, and $[b^m_U] \sub [R(U,A)]$ in $\mathsf{st}(\gen{\B \cup \U})$ for all such $m$. Also, if $x_n \in [b^m_U]$ in $\mathsf{st}(\B)$ then $y_n \in [b^m_U] \sub [R(U,A)]$ in $\mathsf{st}(\gen{\B \cup \U})$. It follows that $Y \cap [R(U,A)] \supseteq X \cap \bigcup_{m \in A}[b^m_U]$ is infinite.
Similarly, there are infinitely many $m \notin A$ such that $[b^m_U] \cap X \neq \0$, and $[b^m_U] \cap [U] = \0$ in $\mathsf{st}(\gen{\B \cup \U})$ for all such $m$. As before, if $x_n \in [b^m_U]$ in $\mathsf{st}(\B)$ then $y_n \in [b^m_U] \sub [R(U,A)]$ in $\mathsf{st}(\gen{\B \cup \U})$, and it follows that $Y \setminus [R(U,A)] \supseteq Y \cap \bigcup_{n \notin A}[b^n_U]$ is infinite.

Hence $[R(U,A)]$ splits $Y$ in $\mathsf{st}(\gen{\B \cup \U})$. But no clopen set can split a convergent sequence; thus, as $[R(U,A)]$ is clopen in $\mathsf{st}(\gen{\B \cup \U})$, $Y$ does not converge in $\mathsf{st}(\gen{\B \cup \U})$.
\end{proof}

\begin{theorem}\label{thm:noseq}
Suppose $\kp$ is a cardinal such that $\sr \leq \kp = \mathrm{cof}(\kp^{\aleph_0},\sub)$. Then $\noseq \,\leq\, \kp$.
\end{theorem}
\begin{proof}
To prove the theorem, we construct via recursion an increasing sequence $\seq{\A_\a}{\a \leq \w_1}$ of $\leq\!\kp$-sized subalgebras of $\mathsf{ro}(2^\w)$. The construction will ensure that $\card{\A_{\w_1}} \leq \kp$ and $\mathsf{st}(\A_{\w_1})$ is non-sequential. This suffices to prove the theorem, because the weight of $\mathsf{st}(\A)$ is $|\A|$.

For the base step of the recursion, let $\A_0 = \mathsf{clop}(2^\w)$, and note that $\A_0$ is countable. 
If $\lambda$ is a limit ordinal, then take $\A_\lambda = \gen{ \bigcup_{\a < \lambda}\A_\a }$, noting that if $\card{\A_\a} \leq \kp$ for every $\a < \lambda$ then $\card{\A_\lambda} \leq \kp$ also.

For the successor step, fix $\a < \w_1$ and suppose $\A_\a$ has already been constructed with $\card{\A_\a} \leq \kp$. Let $\set{C_\xi}{\xi \in \kp}$ be a collection of countable subsets of $\A_\a$ such that every countable subset of $\A_\a$ is contained in some $C_\xi$. (Such a family exists because $\kp = \mathrm{cof}(\kp^{\aleph_0},\sub)$.) For each $\xi \in \kp$, let $\B_\xi = \gen{\mathsf{clop}(2^\w) \cup C_\xi}$. Each $\B_\xi$ is a subalgebra of $\mathsf{ro}(2^\w)$ satisfying the hypothesis of Lemma~\ref{lem:noseq}. Applying the lemma, fix for each $\xi \in \kp$ some $\U_\xi$ with $\card{\U_\xi} \leq \sr$ such that $\B_\xi$ and $\U_\xi$ satisfy property $(*)$. Finally, set $\A_{\a+1} = \gen{ \A_\a \cup \bigcup_{\xi \in \kp}\U_\xi }$ and note that $\card{\A_{\a+1}} \leq \kp$.
This completes the recursion. 

Let $\A = \A_{\w_1}$. 
We claim that $\mathsf{st}(\A)$ is non-sequential. Aiming for a contradiction, suppose $[\mathsf{st}(\A)]^{con} \neq \0$. Let $\seq{y_n}{n \in \N} = Y \in [\mathsf{st}(\A)]^{con}$, and let $y \in \mathsf{st}(\A)$ be the limit point of $Y$. For each $n \in \N$, there is some $D_n \in \A$ such that $Y \cap [D_n] = \{y_n\}$. Let $\a < \w_1$ be big enough so that $\set{C_n}{n \in \N} \sub \A_\a$. 

At stage $\a+1$ of our recursion, there was some $\xi \in \kp$ such that $C_\xi \supseteq \set{D_n}{n \in \N}$, and we found a set $\U_\xi$ such that $\B_\xi = \gen{\mathsf{clop}(2^\w) \cup C_\xi}$ and $\U_\xi$ satisfy $(*)$. For each $n \in \N$, let $x_n = y_n \cap \B_\xi$, and note that the $x_n$ are distinct because $\B_\xi \supseteq \set{D_n}{n \in \N}$. Because $Y$ converges to $y$ in $\mathsf{st}(\A_\a)$, we must also have $\set{x_n}{n \in \N}$ converging to $x = y \cap \B_\xi$. Similarly, for each $n \in \N$, let $z_n = y_n \cap \gen{\B_\xi \cup \U_\xi}$; because $Y$ converges to $y$ in $\mathsf{st}(\A)$, we must also have $\set{x_n}{n \in \N}$ converging to $z = y \cap \gen{\B_\xi \cup \U_\xi}$. As $x_n = \B_\xi \cap z_n$ for each $n$, this contradicts $(*)$.
\end{proof}

\section{Lower bounds for $\sr$}\label{sec:lower}

In the previous section, we showed that $\split$ is a lower bound for $\sr$. In this section we prove that $\cov M$ and $\bdd$ are also lower bounds for $\sr$. 

\begin{theorem}\label{thm:covMlb}
$\cov M \leq \sr$.
\end{theorem}
\begin{proof}
For each $x \in \R$, fix some $S_x \in [\R]^{con}$ converging to $x$. It is not difficult to see that if $U$ is an open subset of $\R$ that splits $S_x$, then $x \in \boundary U$. Now suppose $\U$ is a family of open subsets of $\R$ having the property stated in Definition~\ref{def:sr}. For each $x \in \R$, the set $S_x$ is split by some $U \in \U$, so
$$\R \,=\, \textstyle \bigcup_{U \in \U} \boundary U.$$
But $\boundary U$ is closed and nowhere dense for each open $U \sub \R$, so this shows $\cov M \leq |\U|$. As this is true for every $\U$ having the property stated in Definition~\ref{def:sr}, it follows that $\cov M \leq \sr$.
\end{proof}

In the Cohen model $\split < \cov M$. It follows from this and Theorem~\ref{thm:covMlb} that $\split < \sr$ in the Cohen model. Thus the bound proved in Lemma~\ref{lem:splitlb} is consistently strict. This also shows that the word ``uncountable'' cannot be removed from the hypothesis of Theorem~\ref{thm:Polish}, because $\N$ is a Polish space and $\split(\N) = \split$ is consistently less than $\sr$.

Likewise, $\cov M < \split$ in the Mathias model, and it follows from this and Lemma~\ref{lem:splitlb} that $\cov M < \sr$ in the Mathias model. Thus the bound proved in Theorem~\ref{thm:covMlb} is consistently strict.

Recall that the Cantor space has a basis consisting of sets of the form
$$[s] = \set{x \in 2^\w}{x \rest \mathrm{dom}(s) = s}$$
where $s \in 2^{<\w}$. Here, as usual, $2^{<\w}$ denotes the set of all functions from some finite ordinal $n = \{0,1,\dots,n-1\}$ to $2 = \{0,1\}$.

\begin{theorem}\label{thm:blb}
$\bdd \leq \sr$.
\end{theorem}
\begin{proof}
Suppose $\U$ is a collection of open subsets of the Cantor space $2^\w$ such that $|\U| < \bdd$. To prove the theorem, it suffices to show that there is some infinite $X \sub 2^\w$ that is not split by any $U \in \U$.

For each $s \in 2^{<\w}$, let $s \cat 0^\infty$ denote the member of $2^\w$ defined by setting
$(s \cat 0^\infty)(n) = s(n)$ if $n \in \mathrm{dom}(s)$, and $(s \cat 0^\infty)(n) = 0$ for all $n \geq \mathrm{dom}(s)$.
Likewise, define $s \cat 0^k = (s \cat 0^\infty) \rest (\mathrm{dom}(s)+k)$ for each $k \in \w$; i.e., $s \cat 0^k$ is the finite sequence obtained by appending $k$ zeroes to $s$.

Observe that if $s \in 2^{<\w}$, then the sets of the form $[s \cat 0^k]$ form a neighborhood basis for the point $s \cat 0^\infty$ in $2^\w$.
In particular, if $s \cat 0^\infty \in U \sub 2^\w$ and $U$ is open, then there must be some $k \geq 0$ such that $[s \cat 0^k] \sub U$. 

For each $U \in \U$, define a function $f_U: \w \to \w$ by setting
$$f_U(n) =  \max\set{\min\set{k}{[s \cat 0^k] \sub U}}{s \in 2^{n+1} \, \text{ and } s \cat 0^\infty \in U}$$
for each $n \in \w$. (This function is well-defined by the previous paragraph.)
Equivalently, $f_U(n)$ is the smallest $k$ such that for any $s \in 2^{n+1}$, if $s \cat 0^\infty \in U$ then $[s \cat 0^k] \sub U$.

As $|\U| < \bdd$, there is some $f \in \w^\w$ such that $f_U <^* f$ for every $U \in \U$ (where, as usual, $f_U <^* f$ means that $f_U(n) < f(n)$ for all but finitely many $n \in \w$).
Using recursion, define an infinite (strictly increasing) sequence $\seq{k_n}{n < \w}$ of natural numbers as follows. Let $k_0 = 0$, and for $n \geq 0$ let 
$$k_{n+1} = k_n + 1 + f(k_n).$$

We now define an infinite subset of $2^\w$ that (we claim) is not split by any $U \in \U$.
Given $n \in \w$, define $s_n \in 2^{<\w}$ as follows: $\mathrm{dom}(s_n) = k_n+1$ and
$$s_n(i) = 
\begin{cases}
1 & \text{ if } i = k_m \text{ for some } m, \\
0 & \text{ if not.}
\end{cases}$$
Informally, $s_n$ is a finite sequence of $0$'s and $1$'s, with exactly $n+1$ terms equal to $1$, and where these $1$'s appear sparsely, separated by increasingly enormous strings of $0$'s. More specifically, if the $k^{\mathrm{th}}$ value of the sequence is $1$, then it is followed by a length-$f(k)$ string of $0$'s before the next $1$ appears in the sequence.
\begin{align*}
s_0 & = \ 1 \\ 
s_1 & = \ \underbrace{1 \overbrace{0\,0\,\dots \,0\,0}^{f(0)}}_{k_1} 1 \\
s_2 & = \ \underbrace{\overbrace{1 \,0\,0\,\dots \,0\,0}^{k_1} 1 \overbrace{0\,0\,\dots\dots\,0\,0}^{f(k_1)}}_{k_2} 1 \\
s_3 & = \ \underbrace{\overbrace{1 \,0\,0\,\dots \,0\,0 \,1\, 0\,0\,\dots\dots\,0\,0}^{k_2} 1 \overbrace{0\,0\,\dots\dots\dots\dots\,0\,0}^{f(k_2)}}_{k_3} 1 \\
& \quad \tiny \vdots \\
s_n & = \ 1 \overbrace{0\,0\,\dots \,0\,0}^{f(0)} 1 \overbrace{0\,0\,\dots.\,\,0\,0}^{f(k_1)} 1 \overbrace{0\,0\,\dots.\,.\,\,0\,0}^{f(k_2)} 1\, 0\,0\, \dots\, 0\,0\, 1 \overbrace{0\,0\,\dots\dots\,0\,0}^{f(k_n)} 1 
\end{align*}
We claim that the set $X = \set{s_n \cat 0^\infty}{n \in \w}$ is not split by any $U \in \U$.

To see this, let us suppose that $U \in \U$ contains $s_n \cat 0^\infty$ for infinitely many $n \in \w$. Recall that the sequence $\seq{k_n}{n \in \w}$ is strictly increasing (and in particular, it contains infinitely many numbers); as $f_U <^* f$, this implies that $f_U(k_n) < f(k_n)$ for all but finitely many $n$. Thus there is some $n \in \w$ such that $s_n \cat 0^\infty \in U$ and $f_U(k_n) < f(k_n)$.  

Now $\mathrm{dom}(s_n) = k_n+1$, so by the definition of $f_U$, $s_n \cat 0^\infty \in U$ implies that $[s_n \cat 0^{f_U(k_n)}] \sub U$. As $f(k_n) > f_U(k_n)$, we also have $[s_n \cat 0^{f(k_n)}] \sub [x_n \cat 0^{f_U(n)}]$. Thus $[s_n \cat 0^{f(k_n)}] \sub U$.

But for all $m \geq n$, $s_n \cat 0^{f(k_n)}$ is an initial segment of $s_m$. This implies that $s_m \cat 0^\infty \in [s_n \cat 0^{f(k_n)}] \sub U$ for all $m \geq n$. 

Thus if $U \in \U$ contains infinitely many points of $X$, then it contains co-finitely many points of $X$. Hence $X$ is not split by any $U \in \U$.
\end{proof}

In the Cohen model, $\bdd < \cov M$. It follows from this and Theorem~\ref{thm:covMlb} that $\bdd < \sr$ in the Cohen model. Thus the bound proved in Theorem~\ref{thm:blb} is consistently strict.

Unlike $\split$ and $\cov M$, which are lower bounds for $\noseq$ as well as for $\sr$, $\ \ \ \bdd$ is not a lower bound for $\noseq$. This follows from a result of the second author (to appear in a forthcoming paper) which shows that $\noseq = \aleph_1$ in the Laver model; as $\bdd = \aleph_2$ in the Laver model, this shows the consistency of $\noseq < \bdd$. It also shows that $\noseq < \sr$ in the Laver model, which is interesting in light of the corollary to Theorem~\ref{thm:noseq} mentioned in the introduction.

\section{An upper bound for $\sr$}\label{sec:upper}

In this section we prove our only nontrivial upper bound for $\sr$, namely $\sr \leq \max\{ \bdd, \non N \}$.
Recall that \non N denotes the smallest size of a non-Lebesuge-measurable subset of $2^\w$. (The value of \non N remains the same if instead of $2^\w$ we were to use the Lebesgue measure on $\R$ or $[0,1]$, or any other standard measure on a Polish space; see \cite[Theorem 17.41]{Kechris}). Recall also that if $X \sub 2^\w$ then the \emph{outer measure} of $X$ is
$$\mu^*(X) = \inf \set{ \mu(Y) }{ Y \sub 2^\w \text{ is measurable and } X \sub Y }$$
(see \cite[Chapter 3]{Oxtoby} or \cite[Section 17.A]{Kechris}). In particular, if $X \sub 2^\w$ has outer measure $1$, then it has non-empty intersection with every measurable, non-null subset of $2^\w$.

\begin{lemma}\label{lem:outermeasure1}
\non N is the smallest cardinality of a subset of $2^\w$ with outer measure $1$.
\end{lemma}
\begin{proof}
Let $X \sub 2^\w$ be a non-measurable set with $|X| = \non N$, and let
$$Y = \set{y \in 2^\w}{ \text{there is some } x \in X \text{ such that } x =^* y}.$$
where, as usual, $x =^* y$ means that $x(n) = y(n)$ for all but finitely many $n \in \w$. It is clear that $|Y| = |X| \cdot \aleph_0 = |X|$, and we claim that $\mu^*(Y) = 1$. To see this, suppose $B \sub 2^\w \setminus Y$ is measurable and $\mu(B) > 0$. Let
$$C = \set{c \in 2^\w}{ \text{there is some } b \in B \text{ such that } b =^* c}$$
and note that because $B \cap Y = \0$, we must also have $C \cap Y = \0$. By Kolmogorov's $0$-$1$ Law (also known as the $0$-$1$ Law for Lebesgue measure; see \cite[Exercise 17.1]{Kechris}), because $C$ is closed under the equivalence relation $=^*$, either $\mu(C) = 0$ or $\mu(C) = 1$. The former is impossible because $C \supseteq B$ and $\mu(B) > 0$, so $\mu(C) = 1$. As $Y \cap C = \0$, this implies $Y$ is null, a contradiction.
Thus if $B \sub 2^\w$ is measurable and $Y \sub 2^\w \setminus B$, then $B$ is null. It follows that $\mu^*(Y) = 1$.
\end{proof}

\begin{theorem}\label{thm:upperbound}
$\sr \leq \max\{ \bdd, \non N \}$.
\end{theorem}

\begin{proof}
To prove the theorem, we will show that there is a family $\U$ of open subsets of the Cantor space $2^\w$, with $|\U| = \max\{\bdd,\non N\}$, such that every $X \in [2^\w]^{con}$ is split by some $U \in \U$.

Let us begin with a brief sketch of the proof, and in particular of how the members of $\U$ are obtained. Given some function $f: \w \to \w$, consider the following randomized, $\w$-step algorithm for obtaining an open subset of $2^\w$. On step $0$, randomly select $\frac{1}{4}$ of the nodes from level $f(0)$ of the tree $2^{<\w}$ (such that every node is equally likely to be selected); on step $1$, randomly select $\frac{1}{8}$ of the nodes from level $f(1)$ that do not extend an already-chosen node; on step $2$, randomly select $\frac{1}{16}$ of the nodes from level $f(2)$ that do not extend an already-chosen node; etc. Once a set $S \sub 2^{<\w}$ is selected in this way, define $U = \bigcup_{s \in S}[s]$. This $U$ is a ``random'' open subset of $2^\w$ with measure $\frac{1}{4}+\frac{1}{8}+\dots = \frac{1}{2}$. We shall prove that if $f: \w \to \w$ grows fast enough, then a set $U$ chosen in this way will split a given $X \in [2^\w]^{con}$ with probability $\frac{1}{2}$. Roughly, we need only $\bdd$ functions to make sure that one of them always grows ``fast enough'', and only \non N sequences of selections to ensure that one of them is always ``random enough'' for this idea to work. By fixing these $\bdd$ functions and $\non N$ sequences of selections beforehand, we obtain the desired family $\U$ of open sets.

To begin, we describe a way of associating to every increasing function $f: \w \to \w$ a Polish space $X_f$ and a probability measure $\mu_f$ on $X_f$. Roughly, the idea behind the definition of $X_f$ and $\mu_f$ is that a point of $X_f$ selected randomly with respect to $\mu_f$ corresponds to a ``random'' selection of nodes in $2^{<\w}$ as described in the previous paragraph.

Given $n \in \w$, let $2^{[n]}$ denote the set of functions $\{0,1,\dots,n-1\} \to \{0,1\}$. (This set is often denoted $2^n$, but we adopt a different notation here to distinguish between sets of sequences $2^{[n]}$ and natural numbers of the form $2^n$, both of which appear frequently in the following argument.) Fix an increasing function $f: \w \to \w$ with $f(0) \geq 2$, and for each $n \in \w$ define 
$$\textstyle D^f_n = \set{A \sub 2^{[f(n)]}}{|A| = \frac{1}{2^{n+2}}\card{2^{[f(n)]}} = 2^{f(n) - n - 2}}.$$
In other words, each $A \in D^f_n$ represents a possible outcome of the $n^{\mathrm{th}}$ stage of the randomized selection process described above. The requirement that $f$ be increasing with $f(0) \geq 2$ simply ensures that $|A| = 2^{f(n) - n - 2}$ is always an integer.

Let us consider each $D^f_n$ as a (finite) topological space, having the discrete topology. The space $X_f$ is defined as a subspace of $\prod_{n \in \w}D^f_n$ (which is given the usual product topology) as follows:
\begin{align*}
X_f \,=\, \textstyle \left\{ z \in \prod_{n \in \w}D^f_n \right.\,:\ & \text{if } m < n \text{ then no member of } z(n)  \\ 
& \left. \quad \ \text{ extends any member of } z(m) \vphantom{\textstyle \prod_{n < \w} D^f_n} \right\}.
\end{align*}

We claim that $X_f$ is a Polish space. To see this, first observe that $\prod_{n \in \w}D^f_n$ is a Polish space; in fact, it is homeomorphic to the Cantor space, because it is a countably infinite product of discrete spaces each having more than one point. 
The space $\prod_{n \in \w}D^f_n$ has a basis consisting of sets of the form
$$\textstyle [\![\z]\!] = \set{z \in \prod_{n \in \w}D^f_n\,}{z \rest n = \z}$$
where $\z$ is a function such that $\mathrm{dom}(\z) \in \w$ and $\z(i) \in D^f_i$ for all $i \in \mathrm{dom}(\z)$.
Next, observe that for every pair of natural numbers $m < n$,
\begin{align*}
U_{m,n} \,=\,  & \textstyle \left\{ z \in \prod_{n \in \w}D^f_n \right.\,:\ \text{no member of } z(n)  \\ 
& \left. \qquad \quad \text{ extends any member of } z(m) \vphantom{\textstyle \prod_{n < \w} D^f_n} \right\}
\end{align*}
is open in $\prod_{n \in \w}D^f_n$. (It is a union of some subset of the finitely many basic open sets $[\![\z]\!]$ having $\mathrm{dom}(\z) = n+1$.) Thus $X_f = \bigcap_{m < n < \w}U_{m,n}$ is a $G_\dlt$ subspace of a Polish space, and therefore is itself Polish, by Alexandroff's theorem \cite[Theorem 3.11]{Kechris}.

We call a function $\s$ \emph{admissible} if $\z$ is a function on some ordinal $n$ such that $\z(i) \in D^f_i$ for all $i < n$, and if $i < j < n$ then $\z(j)$ does not contain an extension of any member of $\z(i)$. In other words, $\z$ is admissible if $\z = z \rest n$ for some $z \in X_f$ and $n \in \w$.

\begin{claim1}
If $\z,\xi$ are admissible and $\mathrm{dom}(\z) = \mathrm{dom}(\xi) = n$, then $\z$ and $\xi$ have the same number of admissible extensions with domain $n+1$.
\end{claim1}
\begin{proof}[Proof of claim]
Given an admissible function $\zeta$ and some $k \geq \mathrm{dom}(\zeta)$, we say that $s \in 2^{[f(k)]}$ is \emph{$\zeta$-available} if $s$ does not extend any member of $\zeta(i)$ for any $i < \mathrm{dom}(\zeta)$.

We prove, by induction on $n$ (for all $k$ simultaneously), that
\begin{itemize}
\item[$(*)$] if $\mathrm{dom}(\zeta) = n \leq k$, then the number of $\zeta$-available vertices in $2^{[f(k)]}$ is exactly $\left( \frac{1}{2} + \frac{1}{2^{n+1}} \right) 2^{f(k)}$. 
\end{itemize} 

The base case $n=0$ is trivially true, because if $\mathrm{dom}(\z) = 0$ then every vertex in $2^{[f(k)]}$ is $\zeta$-available, and this is exactly $2^{f(k)} = \left( \frac{1}{2} + \frac{1}{2^{0+1}} \right) 2^{f(k)}$ vertices. Suppose the claim is true for some $m$. Let $\zeta$ be an admissible function with domain $m+1$, and suppose $k \geq m+1$. Note that each of the $\frac{1}{2^{m+2}}2^{f(m)}$ vertices in $\zeta(m)$ has exactly $2^{f(k) - f(m)}$ extensions in $2^{[f(k)]}$, and that two different vertices from $\zeta(m)$ have no common extensions in $2^{[f(k)]}$. 
The $\zeta$-available members of $2^{[f(k)]}$ are exactly the $(\zeta \rest m)$-available vertices minus the extensions of members of $\zeta(m)$.
Thus there are exactly
\begin{align*}
& \left( \frac{1}{2} + \frac{1}{2^{m+1}} \right) 2^{f(k)} - 2^{f(k) - f(m)}|\zeta(m)| \\
& =\, \left( \frac{1}{2} + \frac{1}{2^{m+1}} \right)2^{f(k)} - 2^{f(k) - f(m)}\frac{1}{2^{m+2}}2^{f(m)} \\
& =\, \left( \frac{1}{2} + \frac{1}{2^{m+1}} \right)2^{f(k)} - \frac{1}{2^{m+2}}2^{f(k)} \\
& =\, \left( \frac{1}{2} + \frac{1}{2^{m+2}} \right)2^{f(k)}
\end{align*}
$\zeta$-available vertices in $2^{[f(k)]}$. By induction, $(*)$ holds for all $n \leq k < \w$.

Let us now show that $(*)$ implies the claim.
Let $\zeta$ be an admissible function with domain $n$.
If $M = \left( \frac{1}{2} + \frac{1}{2^{n+1}} \right) 2^{f(n)}$ (the number of $\zeta$-available vertices in $2^{[f(n+1)]}$) and if $N = \frac{1}{2^{n+2}}2^{f(n)}$ (the number of vertices in any member of $D^f_n$) then $\zeta$ has exactly $M \choose N$ admissible extensions $\zeta'$ with domain $n+1$, because $\zeta'(n)$ can be any $N$-sized subset of the $M$ $\zeta$-available vertices in $2^{[f(n+1)]}$. This number depends only on $n$, so this proves the claim.
\end{proof}

If $\z$ is an admissible function, then define
$$\mu_f([\![\z]\!]) = \frac{1}{\card{\set{\xi \,}{\, \xi \text{ is admissible and } \mathrm{dom}(\xi) = \mathrm{dom}(\z)}}}.$$
Using the previous claim, one can show that this function is finitely additive: i.e., if $[\![\z_1]\!], \dots, [\![\z_n]\!]$ are disjoint sets with $\bigcup_{i \leq n}[\![\z_i]\!] = [\![\z]\!]$, then $\mu_f([\![\z]\!]) = \sum_{i \leq n}\mu_f([\![\z_i]\!])$. Thus the equation above defines a measure $\mu_f$ on the semi-ring of all sets of the form $[\![\z]\!]$, where $\z$ is an admissible function. As this semi-ring of sets is a basis for $X_f$, Carath\'eodory's Extension Theorem asserts that $\mu_f$ extends uniquely to a $\s$-additive measure on all Borel subsets of $X_f$.

Note that the empty function $\0$ is admissible, and is the only admissible function with domain $0$. It follows that that $\mu_f(X_f) = \mu_f([\![\0]\!]) = \frac{1}{1} = 1$. Thus $\mu_f$ is a probability measure on $X_f$.

For any well-formed formulas of first-order logic $\varphi(z)$, and $\psi(z)$, define
\begin{align*}
P(\varphi(z)) & \,=\, \mu_f(\set{z \in X_f}{\varphi(z)}), \vphantom{\displaystyle \sum} \\
P(\varphi(z) \,\mid\, \psi(z)) & \,=\, \frac{P(\varphi(z) \wedge \psi(z))}{P(\psi(z))}
\end{align*}
assuming that all these sets are measurable and, in the second equation, that $P(\psi(z)) \neq 0$.
In other words, if $z$ is chosen from $X_f$ randomly with respect to the probability measure $\mu_f$, then $P(\varphi(z))$ is the probability that $z$ satisfies $\varphi(z)$, and $P(\varphi(z) \mid \psi(z))$ is the probability that $z$ satisfies $\varphi(z)$, given that it satisfies $\psi(z)$.

Given $f \in \B$ and $z \in X_f$, define $U^f_{z,n} = \bigcup \set{[s] \sub 2^\w}{s \in z(n)}$ for each $n \in \w$, and define
$$\textstyle U^f_z \,=\, \bigcup_{n \in \w}U^f_{z,n} \,=\, \bigcup \set{[s] \sub 2^\w}{s \in z(n) \text{ for some } n}.$$
(In other words, if $z$ is a ``random'' point of $X_f$, then $U^z$ is the result of the randomized selection process described at the beginning of the proof.)
In keeping with our intended intuition for $X_f$, let us say that a point $z \in X_f$ \emph{selects} a point $x \in 2^\w$ if $x \in U^f_{z}$. We say that $z$ selects $x$ \emph{before stage} $n$ if $x \in \bigcup_{i < n}U^f_{z,i}$; we say that $z$ selects $x$ \emph{at stage} $n$ if $x \in U^f_{z,n}$, and we say that $z$ selects $x$ \emph{after stage} $n$ if $x \in \bigcup_{i > n}U^f_{z,i}$. 

Given two distinct points $x,y \in 2^\w$, let
$$\mathrm{dif}(x,y) = \min\set{n \in \w}{x(n) \neq y(n)}.$$

\vbox{
\begin{claim2} $\ $
\begin{enumerate}
\item For any $x \in 2^\w$, 
$\textstyle P(z \text{ selects } x) = \frac{1}{2}.$ $\vphantom{\displaystyle \sum}$
\item Let $k \in \w$ and let $x,y \in 2^\w$ with $\mathrm{dif}(x,y) \geq f(k)$. Then
$$P (z \text{\emph{ selects }} x \,\mid\, z \text{\emph{ does not select }} y) \,<\, \frac{1}{2^k}.$$
\item Let $k \in \w$, let $x_1, \dots, x_n \in 2^\w$, and suppose that $\mathrm{dif}(x_i,x_j) < f(k)$ for all $1 \leq\, i,j \,\leq n$. Then
\begin{align*}
P \!\left( \textstyle \bigwedge_{1 \leq i \leq n} \ z \text{\emph{ does not select }} x_i \right) 
\,<\, e^{-n/2^{k+2}}.
\end{align*}
\end{enumerate}
\end{claim2}
}
\begin{proof}[Proof of claim]
Let $k \in \w$ and let $\mathrm{Ad}(k)$ denote the set of admissible functions with domain $k$. To begin the proof, we show that if $n < k$ and $s \in 2^{[f(n)]}$, then
\begin{equation}\tag{$\dagger$}
\textstyle \card{\set{\z \in \mathrm{Ad}(k)}{s \in \z(n)}} \,=\, \frac{1}{2^{n+2}}\card{\mathrm{Ad}(k)}.
\end{equation}

To see this, let $t \in 2^{[f(n)]}$ with $t \neq s$. Let $\mathrm{dif}(s,t) = \min\set{i}{s(i) \neq t(i)}$, and let $\pi$ be the permutation of $2^{<\w}$ defined by setting $\pi(r)(i) = r(i)$ for all $i \neq \mathrm{dif}(s,t)$, and otherwise for $\mathrm{dif}(s,t) = d$ we take
$$\pi(r)(d) = 
\begin{cases}
t(d) & \text{ if } r(d) = s(d), \\
s(d) & \text{ if } r(d) = t(d), \\
r(d) & \text{ otherwise}
\end{cases}$$
(Roughly, $\pi$ is the permutation of $2^{<\w}$ obtained by swapping the two cones $s\!\uparrow\ = \set{r \in 2^{<\w}}{s \leq r}$ and $t\!\uparrow\ = \set{r \in 2^{<\w}}{s \leq r}$ in the natural way.)

Given a function $\zeta$ with domain $k$, define a function $\pi\zeta$ by putting
$$r \in \pi\zeta(i) \quad \Leftrightarrow \quad \pi(r) \in \zeta(i)$$
for all $i < k$.
By the definition of $\mathrm{dif}(s,t)$ and $\pi$, it is not hard to see that if $m < n < k$, then some member of $\zeta(n)$ extends some member of $\zeta(m)$ if and only if some member of $\pi\zeta(n)$ extends some member of $\pi\zeta(m)$. 
Thus a function $\zeta$ with domain $k$ is admissible if and only if $\pi\zeta$ is also admissible. In other words, the mapping $\zeta \mapsto \pi\zeta$ is a bijection from $\set{\z \in \mathrm{Ad}(k)}{s \in \z(n)}$ to $\set{\z \in \mathrm{Ad}(k)}{t \in \z(n)}$, showing that
$$\textstyle \card{\set{\z \in \mathrm{Ad}(k)}{s \in \z(n)}} = \textstyle \card{\set{\z \in \mathrm{Ad}(k)}{t \in \z(n)}}.$$

As $t$ was arbitrary, this shows that $\card{\set{\z \in \mathrm{Ad}(k)}{s \in \z(n)}}$ depends only on $n$, and not on $s$. In particular, we have
$$\textstyle 2^{f(n)}\card{\set{\z \in \mathrm{Ad}(k)}{s \in \z(n)}} \,=\, \sum_{\z \in \mathrm{Ad}(k)}|\z(n)|$$
for any $s \in 2^{f(n)}$.
As exactly $\frac{1}{2^{n+2}}2^{f(n)}$ members of $2^{[f(n)]}$ appear in each set of the form $\z(n)$, $\z \in \mathrm{Ad}(k)$, we also have 
$$\textstyle \sum_{\z \in \mathrm{Ad}(k)}|\z(n)| \,=\, \frac{1}{2^{n+2}}2^{f(n)} \card{\mathrm{Ad}(k)}$$
Hence $\textstyle 2^{f(n)}\card{\set{\z \in \mathrm{Ad}(k)}{s \in \z(n)}} = \frac{1}{2^{n+2}}2^{f(n)} \card{\mathrm{Ad}(k)}$
for any $s \in 2^{[f(n)]}$, and dividing both sides by $2^{f(n)}$ gives $(\dagger)$.

Next, we claim that for any $x \in 2^\w$ and $k \in \w$,
\begin{equation}\tag{$\ddagger$}
P(z \text{ selects } x \text{ at stage } k) \,=\, \frac{1}{2^{k+2}}.
\end{equation}

Let $x \in 2^\w$ and $k \in \w$, and observe that if $k > n$ then 
\begin{align*}
\set{z \in X_f}{z \text{ selects } x \text{ at stage } k} & \\
= \textstyle \bigcup & \set{[\![\z]\!]\,}{\z \in \mathrm{Ad}(k+1) \text{ and } x \rest f(k) \in \z(k)}
\end{align*}
and this is a disjoint union.
Because $\mu_f([\![\z]\!]) = \frac{1}{|\mathrm{Ad}(k+1)|}$ for all $\z \in \mathrm{Ad}(k+1)$,
\begin{align*}
P(z \text{ selects } x \text{ at stage } k) & = \textstyle \mu_f(\set{z \in X_f}{z \text{ selects } x \text{ at stage } k}) \\
& = \textstyle \frac{1}{|\mathrm{Ad}(k+1)|}\card{\set{\z \in \mathrm{Ad}(k+1)}{x \rest f(k) \in \z(k)}}.
\end{align*}
Combining this with $(\dagger)$ gives $(\ddagger)$.

Using $(\ddagger)$, let us now prove $(1)$ and $(2)$. For $(1)$, note that
$$\textstyle \set{z \in X_f}{z \text{ selects } x} = \bigcup_{k \in \w}\set{z \in X_f}{z \text{ selects } x \text{ at stage } k}$$
and that this is a disjoint union, so
\begin{align*}
P(z \text{ selects } x) & = \textstyle \sum_{k \in \w} P(z \text{ selects } x \text{ at stage } k) \\
& = \textstyle \sum_{k \in \w}\frac{1}{2^{k+2}} = \frac{1}{2}.
\end{align*}
For $(2)$, note that if $x,y \in 2^\w$ and $f(k) \geq \mathrm{dif}(x,y)$, then
\begin{align*}
P & (z \text{ selects } x \text{ and does not select } y) \\
& \qquad <\, P (z \text{ selects } x \text{, but not before stage } k ) \\
& \qquad =\, \textstyle \sum_{n \geq k} P( z \text{ selects } x \text{ at stage } n ) \\
& \qquad =\, \textstyle \sum_{n \geq k} \frac{1}{2^{n+2}} \,=\, \displaystyle \frac{1}{2^{k+1}}.
\end{align*}
Combining this with $(1)$, we have
\begin{align*}
P (z \text{ selects } x & \,\mid\, z \text{ does not select } y) \\
& =\, \frac{P(z \text{ selects } x \text{ and does not select } y)}{P(z \text{ does not select } y)} \\
& =\, \frac{P(z \text{ selects } x \text{ and does not select } y)}{1 - P(z \text{ selects } y)} \,<\, \frac{\frac{1}{2^{k+1}}}{1-\frac{1}{2}} \,=\, \frac{1}{2^k}.
\end{align*}

To prove $(3)$, first observe that 
\begin{align*}
P \!\left( \textstyle \bigwedge_{1 \leq i \leq n} \right.&\left. z \text{ does not select } x_i \vphantom{\textstyle \bigwedge_{1 \leq i \leq n}} \right) \\
\leq\, P \!\left(\vphantom{\textstyle \bigwedge_{1 \leq i \leq n}} \right. & \left. \textstyle \bigwedge_{1 \leq i \leq n} \ z \text{ does not select } x_i \text{ at or before stage } k \right) \\
& \leq\ \frac{P \!\left(\vphantom{\textstyle \bigwedge_{1 \leq i \leq n}} \textstyle \bigwedge_{1 \leq i \leq n} \ z \text{ does not select } x_i \text{ at or before stage } k \right)}{P \!\left(\vphantom{\textstyle \bigwedge_{1 \leq i \leq n}} \textstyle \bigwedge_{1 \leq i \leq n} \ z \text{ does not select } x_i \text{ before stage } k \right)} \\
& =\ P \!\left( \textstyle \bigwedge_{1 \leq i \leq n} \ z \text{ does not select } x_i \text{ at stage } k \right.  \,\Big|\, \\
& \qquad \qquad \qquad  \left. \textstyle  \bigwedge_{1 \leq i \leq n} \ z \text{ does not select } x_i \text{ before stage } k \right) 
\end{align*}
For convenience, let $\a(z)$ denote the statement ``$z$ does not select any of the $x_i$ at stage $k$'' and let $\b(z)$ denote the statement ``$z$ does not select any of the $x_i$ before stage $k$.''
According to the inequalities above, to prove $(3)$ it suffices to show that
\begin{align*}
P \!\left( \a(z) \,\mid\, \b(z) \right) \,<\, \displaystyle e^{-n/2^{k+2}}.
\end{align*}

For each $i \leq n$, let $s_i = x_i \rest (f(k)+1)$ and note that $s_i \neq s_j$ whenever $i \neq j$, because (by assumption) $\mathrm{dif}(x_i,x_j) < f(k)$ whenever $i \neq j$. 

Given an admissible function $\zeta$ and some $\ell \geq \mathrm{dom}(\zeta)$, let us say (as in the proof of Claim 1) that $s \in 2^{[f(\ell)]}$ is \emph{$\zeta$-available} if $s$ does not extend any member of $\zeta(i)$ for any $i < \mathrm{dom}(\zeta)$.
We showed in the proof of Claim 1 that
\begin{itemize}
\item[$(*)$] if $\mathrm{dom}(\zeta) = k \leq \ell$, then the number of $\zeta$-available vertices in $2^{[f(\ell)]}$ is exactly $\left( \frac{1}{2} + \frac{1}{2^{k+1}} \right) 2^{f(\ell)}$. 
\end{itemize} 
If $\z$ is admissible and $\mathrm{dom}(\z) = k$, then define
\begin{align*}
\mathrm{Ext}_{k+1}(\zeta) & = \set{\zeta'}{\zeta' \text{ is admissible, } \mathrm{dom}(\zeta') = k+1, \text{ and } \zeta' \rest k = \zeta}, \\
\mathrm{Ext}^\a_{k+1}(\zeta) & = \set{\zeta' \in \mathrm{Ext}_{k+1}(\zeta)}{\text{for every } i \leq n,\, s_i \notin \zeta'(k)}.
\end{align*}

If $\zeta' \in \mathrm{Ext}_{k+1}(\zeta)$, then (by definition) $\zeta'(k)$ can be any $\frac{1}{2^{k+2}}2^{f(k)}$-sized subset of the $\left( \frac{1}{2} + \frac{1}{2^{k+1}} \right)2^{f(k)}$ $\zeta$-available vertices in $2^{[f(k)]}$. In other words, setting $M = \left( \frac{1}{2} + \frac{1}{2^{k+1}} \right)2^{f(k)}$ and $N = \frac{1}{2^{k+2}}2^{f(k)}$ we have 
$$\card{\mathrm{Ext}_{k+1}(\zeta)} = {M \choose N} \quad \text{ and } \quad \card{\mathrm{Ext}^\a_{k+1}(\zeta)} = {M-n \choose N}.$$
We now proceed to find an upper bound for the ratio
$$\frac{\card{\mathrm{Ext}^\a_{k+1}(\zeta)}}{\card{\mathrm{Ext}_{k+1}(\zeta)}} \,=\, \frac{{M-n \choose N}}{{M \choose N}}.$$
Observe that
\begin{align*}
\frac{{M-n \choose N}}{{M \choose N}} \, & =\, \frac{\ \ \frac{(M-n)(M-n-1) \cdots (M-N-n+1))}{(N)(N-1) \cdots (2)(1)}\ \ }{\frac{M(M-1) \cdots (M-N+1)}{N(N-1) \dots (2)(1)}} \\
& =\, \frac{(M-n)(M-n-1) \cdots (M-N-n+1)}{M(M-1) \cdots (M-N+1)} \\
& \leq\, \left( \frac{M-n}{M} \right)^{\!N} \,=\, \left( 1 - \frac{n}{M} \right)^{\!N}.
\end{align*}
Next, recall that if $0 < x < 1$ then $\ln (1 - x) < -x$. (One can see this, for example, by using Taylor series: if $0 < x < 1$ then we have $\ln (1-x) = \textstyle -\sum_{m = 1}^\infty \frac{x^m}{m} \,<\, \displaystyle -x.$) Therefore
$$\ln \!\left( \frac{{M-n \choose N}}{{M \choose N}} \right) \,\leq\, \ln \!\left( 1 - \frac{n}{M} \right)^{\!N} \,=\, N \ln \!\left( 1 - \frac{n}{M} \right) \,<\, \frac{-nN}{M}$$
and it follows that $\ln \!\left( \frac{\card{\mathrm{Ext}^\a_{k+1}(\zeta)}}{\card{\mathrm{Ext}_{k+1}(\zeta)}} \right) < \displaystyle \frac{-nN}{M}$ or, equivalently,
$$\frac{\card{\mathrm{Ext}^\a_{k+1}(\zeta)}}{\card{\mathrm{Ext}_{k+1}(\zeta)}} \,<\, e^{\frac{-nN}{M}}.$$
Now observe that
$$\frac{N}{M} \,=\, \frac{\frac{1}{2^{k+2}}2^{f(k)}}{\left( \frac{1}{2} + \frac{1}{2^{k+1}} \right)2^{f(k)}} \,>\, \frac{1}{2^{k+2}}$$
and putting this together with the previous inequality gives
$$\frac{\card{\mathrm{Ext}^\a_{k+1}(\zeta)}}{\card{\mathrm{Ext}_{k+1}(\zeta)}} \,<\, e^{-n/2^{k+2}}.$$

To finish the proof, let
\begin{align*}
B = \set{\zeta}{\zeta \text{ is admissible, } \right.& \mathrm{dom}(\zeta) = k, \text{ and }\\
&\left. x_i \rest j \notin \zeta(j) \text{ for any } j < k \text{ and } 1 \leq i \leq n}.
\end{align*}
Observe that $\b(z)$ holds if and only if $z \rest k \in B$; that is,
$$\set{z \in X_f}{\b(z)} = \textstyle \bigcup \set{[\![\zeta]\!]}{\zeta \in B},$$
and this is a disjoint union. It follows that
$$P(\b(z)) \,=\, \textstyle \mu_f(\set{z}{\b(z)}) \,=\, \sum_{\z \in B}\mu_f([\![\z]\!]).$$

Similarly, observe that $\a(z) \wedge \b(z)$ holds if and only if $z \rest k \in B$ and $z \rest (k+1) \in \mathrm{Ext}_{k+1}^\a$; that is,
$$\set{z \in X_f}{\a(z) \wedge \b(z)} = \textstyle \bigcup \set{[\![\zeta']\!]}{\z' \in \mathrm{Ext}_{k+1}^\a(\z) \text{ for some } \zeta \in B },$$
and this is also a disjoint union. It follows that
$$\textstyle P(\a(z) \wedge \b(z)) \,=\, \mu_f(\set{z}{\a(z) \wedge \b(z)}) \,=\, \sum_{\z \in B}\sum_{\z' \in \mathrm{Ext}^\a_{k+1}(\z)} \mu_f([\![\z']\!]).$$

Finally, note that because of the way we have defined $\mu_f$, if $\z \in B$ then $\mu_f([\![\z']\!]) = \frac{\mu_f([\![\z]\!])}{\card{\mathrm{Ext}_{k+1}(\z)}}$ for all $\z' \in \mathrm{Ext}_{k+1}(\z)$, which implies 
$$\textstyle \sum_{\z' \in \mathrm{Ext}^\a_{k+1}} \mu_f([\![\z']\!]) \,=\, \mu_f([\![\z]\!])\frac{\card{\mathrm{Ext}^\a_{k+1}(\zeta)}}{\card{\mathrm{Ext}_{k+1}(\zeta)}}$$
for any $\z \in B$.
Hence
\begin{align*}
P(\a(z) \,\mid\, \b(z)) & \,=\, \frac{P(\a(z) \wedge \b(z))}{P(\b(z))} \\
& =\, \frac{\sum_{\z \in B} \sum_{\z' \in \mathrm{Ext}^\a_{k+1}(\z)} \mu_f([\![\z']\!])}{\sum_{\z \in B}\mu_f([\![\z]\!])} \\
& =\, \frac{\sum_{\z \in B}\mu_f([\![\z]\!])\frac{\card{\mathrm{Ext}^\a_{k+1}(\zeta)}}{\card{\mathrm{Ext}_{k+1}(\zeta)}}}{\sum_{\z \in B}\mu_f([\![\z]\!])} \\
& <\, \frac{\sum_{\z \in B}\mu_f([\![\z]\!])e^{-n/2^{k+2}}}{\sum_{\z \in B}\mu_f([\![\z]\!])} \,=\, e^{-n/2^{k+2}}
\end{align*}
as claimed.
\end{proof}

Recall that two measure spaces $(X,\mu)$ and $(Y,\nu)$ are \emph{isomorphic} if there is a bijection $\phi: X \to Y$ such that, for every $Z \sub X$, $\phi(Z)$ is $\nu$-measurable if and only if $Z$ is $\mu$-measurable, and if this is the case then $\mu(Z) = \nu(\phi(Z))$. By \cite[Theorem 17.41]{Kechris}, the measure $\mu_f$ defined on $X_f$ above is isomorphic to the Lebesgue measure on $2^\w$. From this and Lemma~\ref{lem:outermeasure1}, it follows that there is a subset $Z_f \sub X_f$ with $|Z_f| = \non N$ such that $Z_f$ has outer $\mu_f$-measure $1$. (If $\phi: 2^\w \to X_f$ is an isomorphism of the Lebesgue measure with $\mu_f$, then we may simply take $Z_f = \phi(Z)$, where $Z \sub 2^\w$ has outer measure $1$ and $|Z| = \non N$.)

Let $\B$ be an unbounded family of functions $\w \to \w$ with $|\B| = \bdd$. Without loss of generality, we may assume that every $f \in \B$ is an increasing function with $f(0) \geq 2$ (so that $X_f$ and $\mu_f$ are well-defined, and have all the properties discussed above). For each $f \in \B$, fix some $Z_f \sub X_f$ such that such that $Z_f$ has outer measure $1$ with respect to $\mu_f$ and $|Z_f| = \non N$. 

Let $\U = \set{U^f_z}{f \in \B \text{ and } z \in Z_f}$, and observe that 
$$|\U| \,=\, \textstyle \sum_{f \in \B}|Z_f| \,=\, |\B| \cdot \non N \,=\, \max\{\bdd,\non N\}.$$
Thus, to finish the proof of the theorem, it suffices to show that every $X \in [2^\w]^{con}$ is split by some $U \in \U$.

Let $X \in [2^\w]^{con}$, and let $p$ denote the (unique) limit point of $X$ in $2^\w$. Replacing $X$ with $X \setminus \{p\}$ if necessary, we may (and do) assume that $p \notin X$. 

For each $i \in \w$, let $N_i$ denote the least natural number with the property that $e^{-N_i/2^{i+2}} < \frac{1}{2^i}$.

Given $F \sub 2^\w$, we say that the points of $F$ are \emph{distinguished by} $\ell \in \w$ if $\mathrm{dif}(x,y) < \ell$ for all $x,y \in F$.
Notice that if the points of $F$ are distinguished by $\ell$, then they are distinguished by any $\ell' > \ell$ also.

Define a function $g_X: \w \to \w$ as follows:
\begin{itemize}
\item for each $k \in \w$, define $g_X(k)$ to be the least natural number such that at least $N_0+N_1+\dots+N_k$ points of $X$ are distinguished by $g_X(k)$.
\end{itemize}
Fix $f \in \B$ such that $f(k) > g_X(k)$ for infinitely many values of $k$. 

\begin{claim3}
$P\!\left( X \setminus U^f_z \text{\emph{ is infinite}} \ \Big|\ z \text{\emph{ does not select }} p \right) = 1$.
\end{claim3}
\begin{proof}[Proof of claim]
Let $n \in \w$, and let $\e > 0$. Fix some $k$ large enough that $\frac{n}{2^k} < \e$, and let $x_1, x_2, \dots, x_n$ be points of $X$ such that $\mathrm{dif}(x_i,p) \geq f(k)$ for all $1 \leq i \leq n$. (Recall that because $p$ is a limit point of $X$, there are $x \in X$ such that $\mathrm{dif}(x,p)$ is arbitrarily large.) Observe that
\begin{align*}
& P \!\left(\textstyle \bigvee_{1 \leq i \leq n} \ z \text{ selects } x_i \ \Big|\, \vphantom{\textstyle \bigwedge_{1 \leq i \leq n}}\!\right. \left. z \text{ does not select } p \vphantom{\textstyle \bigwedge_{1 \leq i \leq n}} \right) \\
& \leq \,
\textstyle \sum_{1 \leq i \leq n} P \!\left( z \text{ selects } x_i \ \Big|\, \vphantom{\textstyle \bigwedge_{1 \leq i \leq n}}\!\right. \left. z \text{ does not select } p \vphantom{\textstyle \bigwedge_{1 \leq i \leq n}} \right) \\
& < \,
\frac{n}{2^k} \,<\, \e.
\end{align*}
(The first inequality follows from the definition of $P(\,\dots)$ and the finite additivity of $\mu_f$, the second follows from part $(2)$ of Claim 2, and the third follows from our choice of $k$.)
It follows that
$$P\!\left( \textstyle \bigwedge_{1 \leq i \leq n} z \text{ does not select }x_i \ \Big|\ z \text{ does not select } p \right) \,>\, 1-\e$$
which implies in particular that
$$P\!\left( X \setminus U^f_z \text{ contains at least } n \text{ points} \ \Big|\ z \text{ does not select } p \right) \,>\, 1-\e.$$
This is true for any $\e > 0$, so
$$P\!\left( X \setminus U^f_z \text{ contains at least } n \text{ points} \ \Big|\ z \text{ does not select } p \right) \,=\, 1.$$
This is true for every $n \in \w$, so
$$P\!\left( X \setminus U^f_z \text{ is infinite} \ \Big|\ z \text{ does not select } p \right) = 1$$
as claimed.
\end{proof}

Notice that the proof of Claim 3 did not require us to use any particular properties of the function $f$; in fact, Claim 3 is true for any $f: \w \to \w$ for which $X_f$ is defined. This is not the case for the next claim, which is only true for sufficiently fast-growing $f$.

\begin{claim4}
$P\!\left( X \cap U^f_z \text{\emph{ is infinite}} \right) = 1$.
\end{claim4}
\begin{proof}[Proof of claim]
By our choice of $f$, there are infinitely many values of $k$ such that at least $N_0+N_1+\dots+N_k$ points of $X$ are distinguished by $f(k)$, where the $N_i$ are defined as above, so that each $N_i$ is big enough that $e^{-N_i/2^{i+2}} < \frac{1}{2^i}$. Let $k_0, k_1, k_2, \dots$ be an infinite, increasing sequence of such values of $k$.

By recursion, define an infinite sequence $X_0, X_1, X_2, \dots$ of finite subsets of $X$ as follows: for each $\ell \in \w$, choose $X_\ell$ so that
\begin{itemize}
\item $|X_\ell| = N_{k_\ell}$,
\item all the points of $X_\ell$ are distinguished by $f(k_\ell)$, and
\item if $j < \ell$ then $X_j \cap X_\ell = \0$.
\end{itemize}
It is clear from our choice of the $k_\ell$ that such a set $X_\ell$ always exists.

Let $\e > 0$, and fix $\ell \in \w$ large enough that $\sum_{j \geq \ell}\frac{1}{2^{k_j}} < \e$. For each $j \geq \ell$, part $(3)$ of Claim 2 states that
\begin{align*}
P \!\left( X_j \cap U^f_z \,=\, \0 \vphantom{U^f_z} \right) & \,=\, P\!\left( \textstyle \bigwedge_{x \in X_j} \ z \text{ does not select } x \right) \\ 
& \,<\, \displaystyle e^{-N_{k_j}/2^{k_j+2}} \,<\, \frac{1}{2^{k_j}}.
\end{align*}
It follows that
$$P\!\left( X_j \cap U^f_z = \0 \text{ for some } j \geq \ell \right) \,\leq\, \sum_{j \geq \ell} P \!\left( X_j \cap U^f_z \,=\, \0 \vphantom{U^f_z} \right) 
 \,<\, \sum_{j \geq \ell}\frac{1}{2^{k_j}} \,<\, \e.$$
Taking complements, we have
$$P\!\left( X_j \cap U^f_z \neq \0 \text{ for all } j \geq \ell \right) \,>\, 1-\e.$$
Because the $X_j$ are disjoint subsets of $X$, this implies
$$P\!\left( X \cap U^f_z \text{ is infinite} \right) \,>\, 1-\e$$
and as $\e$ was arbitrary, this proves the claim.
\end{proof}

From Claims 3 and 4, it follows that
$$P\!\left( U^f_z \text{ splits } X \ \Big|\ z \text{ does not select } p \right) = 1.$$
Observe that if some $z \in X_f$ selects $p$, then $U^f_z \ni p$, which implies (because $U^f_z$ is open and $p$ is the unique limit point of $X$) that $U^f_z$ contains all but finitely many points of $X$. Thus
$$P\!\left( U^f_z \text{ splits } X \ \Big|\ z \text{ selects } p \right) = 0.$$

From part $(1)$ of Claim 2, we also know that $P\!\left( z \text{ does not select } p \right) = {1 - P\!\left( z \text{ selects } p \right)} = \frac{1}{2}$. Therefore
\begin{align*}
& P\!\left( U^f_z \text{ splits } X \right) \\
& \,=\, P\left( \vphantom{U^f_z} z \text{ does not select } p \right)P\!\left( U^f_z \text{ splits } X  \ \Big|\ z \text{ does not select } p \right) \,=\, \frac{1}{2}
\end{align*}
or, in other words, $\set{z}{U^f_z \text{ splits }X}$ has $\mu_f$-measure $\frac{1}{2}$ in $X_f$. Because $Z_f$ has outer $\mu_f$-measure $1$, it follows that $Z_f \cap \set{z}{U^f_z \text{ splits }X} \neq \0$. If $z$ is any point of this intersection, then $U^f_z$ is a member of $\U$ that splits $X$. Thus every member of $[2^\w]^{con}$ is split by some $U \in \U$.
\end{proof}

Theorem~\ref{thm:upperbound} implies that $\sr = \aleph_1$ in the random real model, or a little more generally, in any model obtained by adding random reals to a model of \ch. This implies that such models contain Efimov spaces of weight $\aleph_1$. It is worth mentioning that this was known already for the random real model: it was proved by the second author and David Fremlin, using different methods, that $\mathsf{st}\!\left( (\pnmf)^V \right)$ is an Efimov space in the random real model \cite{Dow&Fremlin}.

\section{Cicho\'n's diagram}\label{sec:cichon}

In the introduction, we claimed that for every cardinal $\kp$ appearing in Cicho\'n's diagram, either $\kp$ is a (consistently strict) lower bound for $\sr$, or $\kp$ is a (consistently strict) upper bound for $\sr$, or else each of $\kp < \sr$ and $\sr < \kp$ is consistent.
Let us review our progress on this so far:

\begin{itemize}
\item $\cov M$ and $\bdd$ are consistently strict lower bounds for $\sr$: That they are lower bounds for $\sr$ was proved in Theorems \ref{thm:covMlb} and \ref{thm:blb}. They are consistently strict because $\bdd < \cov M \leq \sr$ in the Cohen model, and $\cov M < \bdd \leq \sr$ in the Laver model. \vspace{1mm}

\item $\aleph_1$, $\add N$, and $\add M$ are consistently strictly lower bounds for $\sr$: This follows from the previous bullet point, because each of these cardinals is bounded above by both \cov M and $\bdd$. \vspace{1mm}

\item $\continuum$ and $\cof N$ are consistently strict upper bounds for $\sr$: We showed in Theorem~\ref{thm:upperbound} that $\sr \leq \max\{\bdd,\non N\}$. Both $\bdd$ and \non N are bounded above by \cof N and $\continuum$, so it follows that $\sr \leq \cof N,\continuum$. These bounds are consistently strict because in the Miller model, $\aleph_1 = \max\{\bdd,\non N\} < \cof N = \continuum = \aleph_2$. \vspace{1mm}

\item \cov N and \non M are incomparable with $\sr$, in the sense that both $\sr < \cov N, \non M$ and $\sr > \cov N, \non M$ are consistent. To see the first pair of inequalities, note that 
$$\aleph_1 = \bdd = \non N < \cov N = \non M = \continuum$$
in the random model; as $\sr \leq \max\{\bdd,\non N\}$ by Theorem~\ref{thm:upperbound}, it follows that $\sr < \cov N, \non M$ in the random model. To see the second pair of inequalities, note that
$$\aleph_1 = \cov N = \non M < \cov M = \continuum$$
in the Cohen model; as $\cov M \leq \sr$ by Theorem~\ref{thm:covMlb}, we have $\cov N, \non M < \sr$ in the Cohen model.
\end{itemize}

This takes care of $9$ of the $12$ cardinals in Cicho\'n's diagram: we have $5$ consistently strict lower bounds, $2$ consistently strict upper bounds, and $2$ cardinals that are provably incomparable with $\sr$. The remaining $3$ cardinals, \cof M, \non N, and $\dom$, are also incomparable with $\sr$. Here is what we know so far concerning these $3$ cardinals:

\begin{itemize}
\item $\sr < \cof M$ is consistent: Recall that
$\bdd = \non N < \cof M$ in the random model.
As $\sr \leq \max\{\bdd,\non N\}$ by Theorem~\ref{thm:upperbound}, it follows that $\sr < \cof M$ in the random model. \vspace{1mm}

\item $\sr < \dom$ is consistent: Recall that
$\bdd = \non N < \dom$ in the Miller model. 
As $\sr \leq \max\{\bdd,\non N\}$ by Theorem~\ref{thm:upperbound}, it follows that $\sr < \dom$ in the Miller model. \vspace{1mm}

\item $\non N < \sr$ is consistent: Recall that $\non N < \bdd$ in the Laver model. 
As $\bdd \leq \sr$ by Theorem~\ref{thm:blb}, it follows that $\non N < \sr$ in the Laver model. 
\end{itemize}

It remains to prove the consistency of $\sr < \non N$ and the consistency of $\dom,\cof M < \sr$.

\begin{theorem}\label{thm:miller}
It is consistent with \zfc that $\sr < \non N$.
\end{theorem}

The proof strategy is to begin with a model of $\ma+\neg\ch$, and then to force with a $\s$-centered partial order that will make $\sr = \aleph_1$ while preserving the value of $\non N$.

\begin{lemma}\label{lem:miller'sforcing}
There is a $\s$-centered notion of forcing $\PP$ that adds a countable collection $\U$ of open subsets of the Cantor space $2^\w$ such that if $X \in [2^\w]^{con}$ belongs to the ground model, then there is some $U \in \U$ that splits $X$.
\end{lemma}
\begin{proof}
For this lemma we use (a slight variation of) a notion of forcing $\PP$ introduced by Arnie Miller in \cite{Miller}. Conditions in $\PP$ are finite sequences of the form $\seq{(C_i,F_i)}{i < m}$, where each $C_i$ is a clopen subset of $2^\w$ and each $F_i$ is a finite subset of $2^\w \setminus C_i$. A condition $\seq{(C_i',F_i')}{i < n}$ extends the condition $\seq{(C_i,F_i)}{i < m}$ if $m \leq n$ and for each $i < m$, $F_i' \supseteq F_i$ and $C_i' \supseteq C_i$. 
(In Miller's terminology, this poset is called $\PP(2^\w)$.)
Intuitively, this forcing builds countably many open sets $U_i$, $i < \w$, and a condition $\seq{(C_i,F_i)}{i < n}$ constitutes a promise that $C_i \sub U_i$ and that $F_i \cap U_i = \0$ for each $i < n$.

Suppose $G$ is a $V$-generic filter for $\PP$, and in $V[G]$ define
$$U_i = \textstyle \bigcup \set{C}{\text{there is a condition } p \in G \text{ such that } p(i) = (C,F)}$$
for each $i < \w$. We claim that $\U = \set{U_i}{i < \w}$ satisfies the conclusion of the theorem.

Let $X \in [2^\w]^{con}$, $X \in V$, and let $x$ be the unique accumulation point of $X$. There is some $p \in G$ such that for some $i \in \mathrm{dom}(p)$, $p(i) = (C^p_i,F^p_i)$ and $x \in F^p_i$ (because the set of all such $p$ is clearly dense in $\PP$). We claim $U_i$ splits $X$. To see this, define for each natural number $n$ two subsets of $\PP$:
\begin{align*}
D_n & = \set{q \leq p}{C^q_i \text{ contains } \geq\! n \text{ members of }X, \text{ where } q(i) = (C^q_i,F^q_i)}, \\
E_n & = \set{q \leq p}{F^q_i \text{ contains } \geq\! n \text{ members of }X, \text{ where } q(i) = (C^q_i,F^q_i)}.
\end{align*}
$D_n$ is dense below $p$, because for any $p' \leq p$, if $p'(i) = (C^{p'}_i,F^{p'}_i)$ then we may form $q$ by extending $C^{p'}_i$ to include $\geq\!n$ members of the infinite set $X \setminus F^{p'}_i$. 
Likewise $E_n$ is dense below $p$, because for any $p' \leq p$, if $p'(i) = (C^{p'}_i,F^{p'}_i)$ then (because the limit point $x$ of $X$ is not contained in the closed set $C^{p'}_i$) there are infinitely many points in $X \setminus C^{p'}_i$, and we may extend $F^{p'}_i$ to include $\geq\! n$ of them. Thus $G$ meets every $D_n$, from which it follows that $U_i$ contains infinitely many points of $X$, and $G$ meets every $E_n$, from which it follows that $2^\w \setminus U_i$ contains infinitely many points of $X$. Thus $U_i$ splits $X$.
\end{proof}

\begin{lemma}\label{lem:iteratemiller'sforcing}
There is a $\s$-centered notion of forcing $\PP$ that does not change the value of $\continuum$ and that forces $\sr = \aleph_1$.
\end{lemma}
\begin{proof}
Let $\PP$ be the length-$\w_1$, finite support iteration of the notion of forcing described in Lemma~\ref{lem:miller'sforcing}. Let $\U = \bigcup_{\a < \w_1}\U_\a$, where $\U_\a$ denotes the countable set of open subsets of $2^\w$ added at stage $\a$ of the iteration. Note that $|\U| = \aleph_1$. If $X \in [2^\w]^{con}$ in the final model, then because $X$ is countable, there is some $\a < \w_1$ such that $X$ appears at stage $\a$ of the iteration, and then $X$ is split by some $U \in \U_{\a+1}$. Thus in the final model, every $X \in [2^\w]^{con}$ is split by some $U \in \U$, and this shows that $\sr = \aleph_1$.
\end{proof}

\begin{lemma}\label{lem:ma}
Suppose $V \models \ma$ and that $\PP$ is a $\s$-centered notion of forcing in $V$. Then $\PP$ does not lower the value of $\non N$. 
\end{lemma}
\begin{proof}
This fact has been observed before (e.g., it forms part of the proof Corollary 39 in \cite{rearrangement}), but we do not have an exact reference and the proof is short, so we record it here.

Consider, in the extension, any infinite set $A$ of reals with cardinality less than $\continuum^V = \non N^V$; our goal is to prove that $A$ has measure zero. 

A $\s$-centered forcing never adds random reals \cite[Theorem 6.5.31]{B&J},
so no $a \in A$ is random over the ground model.  
Thus, for each $a \in A$ there is a measure-zero Borel set $N_a$ in the ground model such that the canonical extension $\tilde N_a$ with the same Borel code contains $a$.
Because $\PP$ has the countable chain condition, there is, in the ground model, a collection $\C$ of at most $|A|$ Borel sets, each of measure zero, such that all of the $N_a$'s are in $\C$. Because the ground model satisfies \ma and because $|\C| < \continuum$, the ground model has a measure-zero Borel set $N$ that includes all the sets from $\C$ and, in particular, all the $N_a$'s.  

For each $a \in A$, the fact that $N_a \sub N$ is preserved when we pass to the canonical extensions with the same Borel codes in the final model; that is, $\tilde N_a \sub \tilde N$ in $V^{\PP}$.  In particular, $A \subseteq \tilde N$.  But $\tilde N$ has measure zero, because it is the canonical extension of the null set $N$. Hence $A$ has measure zero.
\end{proof}

\begin{proof}[Proof of Theorem~\ref{thm:miller}]
Beginning with a model of $\ma+\neg\ch$, force with the partial order $\PP$ described in Lemma~\ref{lem:iteratemiller'sforcing}. By Lemmas \ref{lem:iteratemiller'sforcing} and \ref{lem:ma}, the resulting model will have $\sr = \aleph_1$ and $\non N = \continuum > \aleph_1$.  
\end{proof}

The cardinal characteristic $\mathfrak{sep}$ was defined by Kamburelis and W\k{e}glorz in \cite{K&W}. We do not record the definition of $\mathfrak{sep}$ here (see \cite{Brendle} instead), but will say simply that $\mathfrak{sep}$ is closely related to $\sr$, being a variant of $\split$ defined in terms of open subsets of $2^\w$, and that it is not difficult to show $\sr \leq \mathfrak{sep}$. We note that in general $\sr \neq \mathfrak{sep}$ because $\non M \leq \mathfrak{sep}$ by \cite[Theorem 1.1]{Brendle}, whereas $\sr < \non M$ in the random model by Theorem~\ref{thm:upperbound}.

In \cite[Section 2]{Brendle}, Brendle proved the consistency of $\dom = \non M < \mathfrak{sep}$. The model he uses is obtained by beginning with $\ma+\neg\ch$ and then performing a length-$\w_1$, finite support iteration of a $\s$-centered notion of forcing $\DD$, defined below, that adds a dominating real. Our proof of Theorem~\ref{thm:Hechler} below uses a similar model. However, proving $\sr = \continuum$ in this sort of model seems more difficult than proving $\mathfrak{sep} = \continuum$, and several new ideas are required. The proof is self-contained insofar as we do not assume the reader is familiar with \cite[Section 2]{Brendle}.

\begin{theorem}\label{thm:Hechler}
It is consistent with \zfc that $\dom = \cof M < \sr$.
\end{theorem}

\begin{proof}
We begin the proof by defining a variant of the Hechler forcing, which we call $\DD$; it is equivalent to the notion of forcing used by Brendle in \cite[Section 2]{Brendle}.
Conditions in $\DD$ are pairs of the form $\< s,\phi \>$, where
\begin{itemize}
\item $s \in \w^{<\w}$, and
\item $\phi$ is a function $\w^{<\w} \to \w$.
\end{itemize}
The ordering on $\DD$ is given by declaring $\< t,\psi \> \leq \< s,\phi \>$ if and only if 
\begin{itemize}
\item $t \supseteq s$, 
\item $\psi(r) \geq \phi(r)$ for all $r \in \w^{<\w}$, and
\item $t(j) \geq \phi(t \rest j)$ for all $j \in \mathrm{dom}(t) \setminus \mathrm{dom}(s)$.
\end{itemize}
If $p = \< s,\phi \> \in \DD$, we write $\stem p = s$.

Note that $\DD$ is $\s$-centered, because any conditions $\< s,\phi_1 \>, \dots, \< s,\phi_n \>$ in $\DD$ with the same stem have a common extension with that stem, namely $\< s,\max\{\phi_1,\dots,\phi_n\} \>$. Also note that forcing with $\DD$ adds a dominating real: if $G$ is $V$-generic for $\DD$, then $\bigcup \set{\stem p}{p \in G}$ is a function $\w \to \w$ that is $\geq^*$ every member of $(\w^\w)^V$.

Ultimately, the theorem will be proved by adding $\aleph_2$ Cohen reals to a model of \ch and then iterating $\DD$ in $\w_1$ steps, with finite support. (In fact, iterating $\DD$ over a model of $\ma+\neg\ch$ or even just $\ma_{ctbl}$, which is equivalent to the equality $\cov M = \continuum$, would suffice. This is the model Brendle uses in \cite[Section 2]{Brendle}, but the present proof is easier if we use the two-step forcing described above.) To begin the proof, though, let us just consider posets of the form $2^{<\w}*\DD^{\a}$, where
\begin{itemize}
\item $2^{<\w}$ denotes, as above, the tree of all finite $0$-$1$ sequences, partially ordered by extension, which is equivalent to the notion of forcing $\cohen{\w}$ for adding one Cohen real.
\item $\DD^{\a}$ denotes the length-$\a$, finite support iteration of $\DD$, where $\a \leq \w_1$. 
\end{itemize}
For each $\a \leq \w_1$, let $\PP_\a = 2^{<\w}*\DD^{\a}$. If $s \in 2^{<\w}$ and $p$ is a $2^{<\w}$-name such that $s \forces p \in \DD^\a$, then we write $s*p$ for the condition $\<s,p\> \in \PP_\a$. If $g \in 2^{\w}$ is $2^{<\w}$-generic over $V$ with $g \supseteq s$, then $\val_g(p)$ denotes the $\DD^\a$-name obtained from $p$ in $V[g]$.

As usual, a name $\dot m$ of a natural number is a \emph{nice} $\PP$-name (for some notion of forcing $\PP$) if it has the form $\dot m = \set{(a,k_a)}{a \in \A}$, where $\A$ is a maximal antichain in $\PP$ and each $k_a$ is a natural number.

For any $s \in 2^{<\w}$ and $j \in \w$, define $s^{\dagger j}$ so that $\mathrm{dom}(s^{\dagger j}) = \mathrm{dom}(s)$ and 
$$s^{\dagger j}(i) = 
\begin{cases}
s(i) & \text{ if } i \neq j, \\
1 - s(j) & \text{ if } i = j.
\end{cases}$$
Similarly, if $s \in 2^\w$ then $s^{\dagger j}$ denotes the function with domain $\w$ given by the above equation. Notice that if $g \in 2^\w$ is $2^{<\w}$-generic over $V$, then so is $g^{\dagger j}$ for all $j \in \w$.

Let us say that a condition $s*p \in \PP_\a$ is \emph{tidy} if 
\begin{itemize}
\item there is some finite $F \sub \a$ such that $s \forces \mathrm{supp}(p) = F$ (i.e., $s$ decides the support of $p$),
\item for every $\b \in F$, if $p(\b) = \< \dot t,\dot \phi \>$ then there is some $t \in \w^{<\w}$ such that $s*(p \rest \b) \forces \dot t = \check t$ (i.e., $s$ decides the stems of the $p(\b)$), and
\item $\1_{2^{<\w}} \forces p \in \DD^\a$ (which is stronger than the requirement $s \forces p \in \DD^\a$, which is necessary to have $s*p \in \PP_\a$ at all).
\end{itemize}
Notice that every condition in $\PP_\a$ has a tidy extension. (This is proved by induction on $\a$ as follows. The base case is trivial. 
If $\a$ is a limit ordinal, then given $s*p \in \PP_\a$ first extend to a condition $s'*p'$ such that $s'$ decides $\mathrm{supp}(p')$, then note that (because $\DD^\a$ uses finite supports) if $s'$ decides $\mathrm{supp}(p')$ then $s'*p' \in \PP_\b$ for some $\b < \a$ and therefore has a tidy extension by the induction hypothesis.
For the successor case, fix $s*p \in \PP_{\a+1}$. First extend $s*(p \rest \a)$ to some $s'*\hat p \in \PP_\a$ that decides $\stem{p(\a)}$, then (using the induction hypothesis) extend $s'*\hat p \in \PP_\a$ to a tidy condition $s''*\hat q$. Then $s''*(\hat q \cat \hat p(\a))$ is a tidy extension of $s*p$.)

For each $j \in \w$, the function $s \mapsto s^{\dagger j}$ is an automorphism of the forcing notion $2^{<\w}$. The function $s*p \mapsto s^{\dagger j}*p$ is not an automorphism of $\PP_\a$, even if we insist that $\1_{2^{<\w}} \forces p \in \DD^\a$ in order to guarantee that $s^{\dagger j}*p \in \PP_\a$. The following claim states that, nonetheless, this function is ``almost'' an automorphism of $\PP_\a$. This claim is the main tool used for proving $\sr = \aleph_1$ in our forcing extension below.

\begin{claim1}
Let $\a \leq \w_1$, and let $s*p \in \PP_\a$ be a tidy condition. Then there is some $s*q \leq s*p$ such that
$$\val_{g^{\dagger j}}(q) \leq \val_g(p)$$
for all $2^{<\w}$-generic $g$ extending $s$ and for all $j \in \w$.
\end{claim1}

When the conclusion of this claim holds, we say that $s*q$ \emph{symmetrically extends} $s*p$.

\begin{proof}[Proof of claim]
The proof is by induction on $\a$. We will prove an ostensibly stronger claim in order to have a sufficiently strong induction hypothesis:

\begin{claim*}
Let $\a \leq \w_1$, and let $s*p \in \PP_\a$ be tidy. Then there is some $s*q \leq s*p$ such that
for any nice $\PP_\a$-name $\dot m$ for a natural number, there is a nice $\PP_\a$-name $\dot M$ for a natural number such that for any $2^{<\w}$-generic $g$ extending $s$ and any $j \in \w$,
\begin{enumerate}
\item $\val_{g^{\dagger j}}(q) \leq \val_g(p)$, and
\item $s*p \forces\, \val_{g^{\dagger j}}(\dot M) \geq \val_g(\dot m)$. 
\end{enumerate}
\end{claim*}

When $(2)$ holds, we say that $\dot M$ \emph{symmetrically dominates} $\dot m$ at $s*q$. 

For the base case $\a = 0$, we have $\PP_0 = 2^{<\w}$. In this case $(1)$ is trivially true and we must show only the second claim.
Let $\dot m = \set{(a_\ell,k_\ell)}{\ell \in \w}$ be a nice $2^{<\w}$-name for a natural number. For now, we work only in the ground model and construct $\dot M$ as follows.

To begin, we claim that for each $r \in 2^{<\w}$ and $j \in \w$, there is some $t \in 2^{<\w}$ such that $t \supseteq r$ and $t^{\dagger j} \supseteq a_\ell$ for some $\ell \in \w$. This is because $\set{a_\ell^{\dagger j}}{\ell \in \w}$ is a maximal antichain in $2^{<\w}$; in particular, given $r \in 2^{<\w}$ there is always some $t \in 2^{<\w}$ such that $t \supseteq r$ and $t \supseteq a_\ell^{\dagger j}$ for some $\ell \in \w$. As $a_\ell^{\dagger j} \sub t$ if and only if $a_\ell \sub t^{\dagger j}$, this $t$ is as required.

Next, we claim that for any $r \in 2^{<\w}$, there is some $t \supseteq r$ and $M \in \w$ such that 
\begin{equation}\tag{$*$}
t^{\dagger j} \forces  \dot m \leq M \qquad \text{for all } j \in \w.
\end{equation}
To see this, fix $r \in 2^{<\w}$. Choose $t_0 \in 2^{<\w}$ such that $t_0 \supseteq r$ and $s_0 \supseteq a_{\ell_0}$ for some $\ell_0 \in \w$. (Note that some such $t_0$ and $\ell_0$ exist, because $\set{a_\ell}{\ell \in \w}$ is a maximal antichain in $2^{<\w}$.) Let $N = \mathrm{dom}(t_0)$. Using recursion, choose an ascending sequence $t_0 \sub t_1 \sub \dots \sub t_N$ of members of $2^{<\w}$ such that for each $j < N$ there is some $\ell_{j+1} \in \w$ with $t_{j+1} \supseteq a_{\ell_{j+1}}^{\dagger j}$. (Note that some such $t_{j+1}$ and $\ell_{j+1}$ always exist, by the previous paragraph.) Let $t = t_N$ and let
$$M = \max \set{k_{\ell_j}}{j \leq N}.$$
To see that $t$ and $M$ are as required, let $j \in \w$. If $j < N$, then
$$t^{\dagger j} \supseteq t_{j+1}^{\dagger j} \supseteq a_{\ell_{j+1}} \forces \dot m = k_{\ell_{j+1}} \leq M.$$
If on the other hand $j \geq N$, then
$$t^{\dagger j} \supseteq t_0^{\dagger j} = t_0 \supseteq a_{\ell_0} \forces \dot m = k_{\ell_0} \leq M.$$
(The equality $t_0^{\dagger j} = t_0$ follows from the fact that $j \geq N \geq \mathrm{dom}(t_0)$.) Either way, $(*)$ holds.

By the previous paragraph, there is a maximal antichain $\set{t_i}{i \in \w}$ of members of $2^{<\w}$, along with integers $M_i$, $i \in \w$, such that each pair $t_i, M_i$ satisfies $(*)$. Let
$$\dot M = \set{(t_i,M_i)}{i \in \w}.$$

It is clear that $\dot M$ is a nice name. To see that $\dot M$ symmetrically dominates $\dot m$ at $s$, suppose $g \supseteq s$ is a $2^{<\w}$-generic real and fix $j \in \w$. Because $\set{t_i}{i \in \w}$ is a maximal antichain in the ground model, there is some $i \in \w$ such that $t_i \sub g^{\dagger j}$ or, equivalently, $t_i^{\dagger j} \sub g$. But $t_i^{\dagger j} \forces \dot m \leq M_i$ and $t_i \forces \dot M = M_i$. This implies $\val_{g^{\dagger j}}(\dot M) = M_i \geq \val_g(\dot m)$.
As $j \in \w$ was arbitrary, $(2)$ holds. (Notice that we have actually proved something a little stronger than $(2)$: that $\dot M$ symmetrically dominates $\dot m$ at $\0$.)

Now fix $\a < \w_1$ and suppose the (revised) claim is true for all $\b < \a$; let us prove it is also true for $\a$. 

For $(1)$, if $\a$ is a limit ordinal then there is nothing to prove: because $\DD^\a$ is a finite support iteration, $s*p \in \DD^\a$ implies $s*p \in \DD^\b$ already for some $\b < \a$, and if $s*q \in \PP_\b$ is such that $\val_{g^{\dagger j}}(q) \leq_{\DD^\b} \val_g(p)$ (which must hold for some $q$ by the inductive hypothesis) then $\val_{g^{\dagger j}}(q) \leq_{\DD^\a} \val_g(p)$ also. 

Suppose $\a = \b+1$ is a successor ordinal, and let us prove that $(1)$ holds for $\a$. Fix $s*p \in \PP_\a$. 
Then $s*(p \rest \b) \forces p(\b) \in \DD$, and in particular, we may find some $\s \in \w^{<\w}$ and a $\PP_\b$-name $\dot \phi$ such that 
$s*(p \rest \b) \forces\ p(\b) = \< \s, \dot \phi \>$. 
By the inductive hypothesis, there is some $\hat q$ such that $s \forces \hat q \in \DD^\b$, $s*\hat q \leq s*(p \rest \b)$, and both $(1)$ and $(2)$ hold for $\hat q$ and $p$.

For each $t \in \w^{<\w}$, let $\dot m_t$ be a nice $\PP_\b$-name for a natural number such that $s*\hat q \forces \dot \phi(t) = \dot m_t$. Using the inductive hypothesis $(2)$: for each $t \in \w^{<\w}$, let $\dot M_t$ be a nice $\PP_\b$-name for a natural number symmetrically dominating $\dot m_t$ at $s*(p \rest \b)$. Let $\dot \psi$ be a $\PP_\b$-name for the function $t \mapsto \dot M_t$. Then define $q \in \PP_{\b+1}$ by setting $q \rest \b = \hat q$ and $q(\b) = \< \check s, \dot \psi \>$. It is clear that this really does define a condition in $\PP_\a$ (because $s*(q \rest \b) \in \PP_\b$ and $s*(q \rest \b) \forces q(\b) \in \DD$). Furthermore, $s*q$ symmetrically extends $s*p$, because if $j \in \w$ then $\val_{g^{\dagger j}}(q \rest \b) = \val_{g^{\dagger j}}(\hat q) \leq_{\PP_\b} \val_g(p \rest \b)$ by our choice of $\hat q$, and $s*(q \rest \b) = s*\hat q \forces q(\b) = \< \check s, \dot \psi \> \leq \< \check s, \dot \phi \> = p(\b)$.

Let $\a \leq \w_1$ and suppose that $(1)$ and $(2)$ hold for every $\b < \a$. By the argument above, this implies that $(1)$ holds for $\a$. To finish our inductive argument, we must show that $(2)$ holds for $\a$ as well.
Let $s*p \in \PP_\a$ and let $\dot m = \set{(a^0_\ell * a^1_\ell,k_\ell)}{\ell \in \w}$ be a nice $\PP_\a$-name for a natural number. (We may index our antichain by $\w$ because $\PP_\a$ has the ccc.)

For each $\ell \in \w$, if $s*p \leq a^0_\ell * a^1_\ell$ then (because $(1)$ holds for $\a$) there is some $s*q \leq s*p$ such that $s*q$ symmetrically extends $s*p$. It follows that, for each $\ell \in \w$, there is a maximal antichain $\A_\ell$ of conditions below $a^0_\ell * a^1_\ell$, each of which symmetrically extends $a^0_\ell * a^1_\ell$. Define
$\dot m' = \set{(a,k_\ell)}{a \in \A_\ell}$.
It is clear that $\1_{\PP_\a} \forces \dot m = \dot m'$. But $\dot m'$ has the property that if $(b^0*b^1,k_\ell) \in \dot m'$ (where $b^0*b_1 \in \A_\ell$), then 
for any $2^{<\w}$-generic $g$ extending $b^0$ and any $j \in \w$,
$\val_{g^{\dagger j}}(b^1) \leq \val_g(a^1_\ell)$.
In particular, if $\val_g(a) \forces \dot m = k$, then this means that there is some $s \in 2^{<\w}, s \sub g$ such that $s*a \forces \dot m = k$; this implies that for any $\PP_\a$-generic $g*G$, if $s*a \in g*G$ then $b^0*b^1 \in g*G$ for some $b^0*b^1 \in \A_\ell$, where $\ell$ is such that $k_\ell = k$; but then by the previous sentence, 
$\val_{g^{\dagger j}}(b^1) \leq \val_g(a^1_\ell) \forces \dot m = \dot m' = k_\ell$
for any $j \in \w$.
Thus $\dot m'$ has the property that 
$$\text{if } \val_g(a) \forces \dot m' = k \text{, then } \val_{g^{\dagger j}}(a) \forces \dot m' = k \text{ for all } j \in \w.$$ 
By replacing $\dot m$ with $\dot m'$ if necessary, we assume in what follows that $\dot m$ has this property.

Now let us work only in the ground model and construct $\dot M$ as follows. 

To begin, we claim that for each $r*r' \in 2^{<\w}*\DD^\a$ and $j \in \w$, there is some $t*t' \in \PP^\a$ and some $\ell \in \w$ such that $t*t' \leq r*r'$ and $t^{\dagger j}*t' \leq a^0_\ell*a^1_\ell$ for all $j \in \w$. 
To see this, note that $\set{a^0_\ell*a^1_\ell}{\ell \in \w}$ is a maximal antichain, so given $r*r'$ we may extend to some $q*q' \leq r*r'$ with $q*q' \leq a^0_\ell*a^1_\ell$ for some $\ell \in \w$; then because $(1)$ holds at $\a$, we may extend further to some $t*t' \leq q*q'$ such that if $g$ is any $2^{<\w}$-generic real with $g \supseteq t \supseteq q$ then $\val_{g^{\dagger j}}(t') \leq \val_g(q') \leq \val_g(a^1_\ell)$, for all $j \in \w$. But this means that
$t^{\dagger j}*t' \leq a^0_\ell*a^1_\ell$ for all $j \in \w$, so that $t*t'$ is as claimed.


Setting $M = k_\ell$ (where $\ell$ is as in the previous paragraph), it follows that for any $r*r' \in \PP_\a$, there is some $t*t' \supseteq r*r'$ and some $M \in \w$ such that 
\begin{equation}\tag{$*$}
t^{\dagger j}*t' \forces  \dot m \leq M \qquad \text{for all } j \in \w.
\end{equation}
Thus there is a maximal antichain $\set{t_i*t_i'}{i \in \w}$ of members of $\PP_\a$, along with integers $M_i$, $i \in \w$, such that each pair $t_i*t_i', M_i$ satisfies $(*)$. Let
$$\dot M = \set{(t_i*t_i',M_i)}{i \in \w}.$$

It is clear that $\dot M$ is a nice name. To see that $\dot M$ symmetrically dominates $\dot m$ at $s*p$, suppose $g \supseteq s$ is a $2^{<\w}$-generic real. Because $\set{t_i}{i \in \w}$ is a maximal antichain for $2^{<\w}$ (otherwise, $\set{t_i*t_i'}{i \in \w}$ could not be a maximal antichain for $\PP_\a$), there is some $i \in \w$ such that $t_i \sub g$ or, equivalently, $t_i^{\dagger j} \sub g^{\dagger j}$ for all $j \in \w$. But  for all $j \in \w$, $t_i^{\dagger j}*t' \forces \dot m \leq M_i$ and $t_i*t' \forces \dot M = M_i$. 
This implies $\val_g(\dot M) \geq \val_{g^{\dagger j}}(\dot m)$ for every $j \in \w$.
\end{proof}


Let $\PP = \cohen{\w_2}*\DD^{\w_1}$. To prove the theorem, we will show that if $\continuum \leq \aleph_2$ then $1_\PP \forces \aleph_1 = \dom = \cof M < \sr = \continuum = \aleph_2$.

\begin{claim2}
Suppose $G$ is $\PP$-generic over $V$. Then $V[G] \models \dom = \cof M = \aleph_1$.
\end{claim2}
\begin{proof}[Proof of claim]
It suffices to show that forcing with $\DD^{\w_1}$ produces a model in which $\dom = \cof M = \aleph_1$ (because forcing with $\PP = \cohen{\w_2}*\DD^{\w_1}$ can be viewed as first forcing with $\cohen{\w_2}$ and afterward forcing with $\DD^{\w_1}$).

Forcing with $\DD^{\w_1}$ makes $\dom = \aleph_1$ because dominating reals are added at every stage of the iteration. Forcing with $\DD^{\w_1}$ makes $\non M = \aleph_1$ because our use of finite supports guarantees that Cohen reals are added at every limit stage of cofinality $\w$. (In fact, Cohen reals are added at every stage of the iteration, because forcing with $\DD$ adds a Cohen real. Specifically, if $d: \w \to \w$ is the dominating real added by $\DD$, then the function $c: \w \to 2$ defined by setting $c(n) = 0$ if and only if $c(n)$ is even is $2^{<\w}$-generic.) Thus $\dom = \non M = \aleph_1$ after forcing with $\DD^{\w_1}$. By a result of Fremlin (see \cite[Theorem 2.2.11]{B&J}), $\cof M = \max\{ \dom,\non M \}$, and so it follows that $\cof M = \aleph_1$ after forcing with $\DD^{\w_1}$.
\end{proof}

\begin{claim3}
Suppose that $V \models \continuum \leq \aleph_2$ and that $G$ is $\PP$-generic over $V$. Then $V[G] \models \sr = \aleph_2$.
\end{claim3}
\begin{proof}[Proof of claim]
Let $\U$ be a collection of open subsets of $2^\w$ in $V[G]$, with $|\U| \leq \aleph_1$. To prove the claim, it suffices to show that there is some $X \in [2^\w]^\infty$ not split by any $U \in \U$.

For each $U \in \U$, let $\dot U$ be a nice $\PP$-name for $U$. Using the fact that $\PP$ has the countable chain condition, there is some collection $\mathcal N$ of nice $\PP$-names with $\mathcal N \in V$ such that $\set{\dot U}{U \in \U} \sub \mathcal N$ and $|\mathcal N| = \aleph_0 \cdot |\U| \leq \aleph_1$.

Each nice $\PP$-name for an open subset of $2^\w$ has countable support; i.e., it is really a $\cohen{S}*\DD^{\w_1}$-name for some countable $S \sub \w_2$. It follows that there is some $\lambda < \w_2$ such that every $\dot V \in \mathcal N$ is really a $\cohen{\lambda}*\DD^{\w_1}$-name. In particular, there is a countable $A \sub \w_2$ such that every $\dot V \in \mathcal N$ is a nice $\cohen{\w_2 \setminus A}*\DD^{\w_1}$-name.

Observe that we may factor $\PP$ by breaking the Cohen part into two pieces along $A$: $\PP = \cohen{\w_2 \setminus A}*\cohen{A}*\DD^{\w_1}$. 

Let $G$ be $\PP$-generic over $V$. Let $G_0$ be the restriction of $G$ to $\cohen{\w_2 \setminus A}$, and let $V_1 = V[G_0]$. Note that $V[G]$ can be viewed as a forcing extension of $V_1$, using the forcing notion $2^{<\w}*\DD^{\w_1}$. Fix $h$ and $H$ such that $h$ is $2^{<\w}$-generic over $V_1$, $H$ is $\DD^{\w_1}$-generic over $V_1[h]$, and $V[G] = V_1[h][H]$. 

We claim that in the final model $V[G]$, no $U \in \U$ splits the set $X = \set{h^{\dagger j}}{j \in \w} \in [2^\w]^{con}$.
To prove this, we will work in $V_1$ and show that every $\dot V \in \mathcal N$ is forced (in $V_1$) not to split $X$. 

Let $\dot g$ denote the standard $V_1$-name for the $2^{<\w}$-generic real, so that every $s \in 2^{<\w}$ forces $s \sub \dot g$. Let $\dot X$ be the natural $V_1$-name for the set $X$, so that $\1_{2^{<\w}} \forces \dot X = \set{(\dot g)^{\dagger j}}{j \in \w}$ and every $s \in 2^{<\w}$ forces $s^{\dagger j} \sub \dot g^{\dagger j}$.

Fix $\dot V \in \mathcal N$ and suppose that $s*p \forces \dot g \notin \dot V$. (Note that if $s*p \forces \dot g \notin \dot V$, then $s*p \forces$ ``$\dot V$ does not split $\dot X$'' because $\dot V$ names an open set, and $\dot X$ names a convergent sequence.)

By Claim 1, there is some $s*q \in 2^{<\w}*\DD^{\w_1}$ that symmetrically extends $s*p$.
Suppose $g$ is any $2^{<\w}$-generic filter containing $s$ (possibly, but not necessarily, equal to $h$). It follows that $g^{\dagger j}$ is also a $2^{<\w}$-generic filter extending $s^{\dagger j}$ for every $j \in \w$. If $j \geq |s|$, then $g$ and $g^{\dagger j}$ are both generic filters extending $s=s^{\dagger j}$.

Because $s*q \leq s*p$, we have $s*q \forces \dot g \notin \dot V$. As $g \supseteq s$ and $g$ is a $2^{<\w}$-generic filter extending $s$, this implies $\val_g(q) \forces g = \val_g(\dot g) \notin \val_g(\dot V)$ in $V_1[g]$. If $j \geq |s|$, then $g^{\dagger j}$ is also a $2^{<\w}$-generic filter extending $s$, and so as before,
$$\val_{g^{\dagger j}}(q) \forces g^{\dagger j} = \val_{g^{\dagger j}}(\dot g) \notin \val_{g^{\dagger j}}(\dot V)$$
in $V_1[g^{\dagger j}]$. 

Now, $V_1[g^{\dagger j}] = V_1[g]$ (because $g$ and $g^{\dagger j}$ are definable from one another) and $\val_{g^{\dagger j}}(\dot V) = \val_g(\dot V)$ (because the support of $\dot V$ is disjoint from $\cohen{A}$). 
%
%
Making these substitutions, we see that $\val_{g^{\dagger j}}(q) \forces g^{\dagger j} \notin \val_g(\dot V)$ in $V_1[g]$. As this holds for any $2^{<\w}$-generic $g$ extending $s$, we conclude that $s*q \forces \dot g^{\dagger j} \notin \dot V$.
As this is true for any $j \geq |s|$, we have
$$s*q \forces \dot g^{\dagger j} \notin \dot V \quad \text{ for all }j > |s|.$$
Thus $s*q \forces$ ``$\dot V$ does not split $\dot X$." 

Thus we see that if $s*p \forces$ ``$\dot g \notin \dot V$'' then some extension $s*q \leq s*p$ forces that $\dot V$ does not split $\dot X$. Thus no $s*p \in 2^{<\w}*\DD^{\w_1}$ can force both ``$\dot g \notin \dot V$'' and ``$\dot V \text{ splits } \dot X$''. On the other hand, if $s*p$ were to force ``$\dot V$ splits $\dot X$'', then it would have to force also that $\dot g \notin \dot V$, because $\dot V$ names an open set, and $\1_{2^{<\w}} \forces \dot X$ converges to $\dot g$. Thus there is no $s*p \in 2^{<\w}*\DD^{\w_1}$ forcing that $\dot V$ splits $\dot X$. 
Consequently, $\val_{G}(\dot V)$ does not split $X$ in $V_1[h][H] = V[G]$. 

This is true for any $\dot V \in \mathcal N$, so it follows that in $V[G]$, no member of $\U$ splits $X$. Because $\U$ was an arbitrary $\aleph_1$-sized collection of open subsets of $2^\w$, $V[G] \models \sr > \aleph_1$. A standard argument shows that $V[G] \models \continuum = \aleph_2$, so $V[G] \models \sr = \aleph_2$ as claimed.
\end{proof}

The theorem follows from the previous two claims: beginning with a model of $\continuum \leq \aleph_2$, forcing with $\PP$ produces a model in which $\dom = \cof M < \sr$.
\end{proof}

\section{Open questions, and the status of Efimov's problem}\label{sec:questions}

As mentioned in the introduction (and proved again in this paper), it is consistent that Efimov spaces exist. Therefore Efimov's problem now takes the form: \emph{Is it consistent that there are no Efimov spaces?}

Our work here does not ``move the needle'' on this problem, in the sense that this paper does not show Efimov spaces to exist in any models of \zfc where they were not proved to exist already. The reason is that the second author proved in \cite{Dow} that if $\split = \mathrm{cof}(\split^{\aleph_0},\sub)$ and $2^\split < 2^\continuum$, then there is an Efimov space. Recently, however, he has realized that the proof in \cite{Dow} does not require the full strength of this hypothesis: it is enough that $\split \leq \kp < \continuum$ for some cardinal $\kp$ such that $\kp = \mathrm{cf}(\kp^{\aleph_0},\sub)$.
In other words, \cite{Dow} shows that one may weaken the hypothesis of Theorem~\ref{thm:noseq} by replacing ``$\sr \leq \kp$'' with ``$\split \leq \kp$,'' but at the cost of weakening the conclusion by replacing ``$\noseq \leq \kp$'' with ``there is an Efimov space." 

The Efimov space constructed in \cite{Dow} has weight $\continuum$, so the construction there gives us no information about $\noseq$. Thus if one is only interested in the existence of Efimov spaces, (s)he can safely ignore our paper. If one is also interested in what properties they can have, such as having small weight, then Theorem~\ref{thm:noseq} has something to say. 

While the Efimov spaces in \cite{Dow} have weight $\continuum$, they are ``small'' in another interesting way: they have character $\split$, which means that each point has a local basis of size $\leq\!\split$. While the weight of an Efimov space is always bounded below by $\split$, we know of no such bound on the character of an Efimov space.

\begin{question}
Is it consistent that an Efimov space has character $<\!\split$?
\end{question}

With Shelah, the second author also proved that $\bdd = \continuum$ implies the existence of an Efimov space \cite{Dow&Shelah}. Thus the status of Efimov's problem can be summed up as follows: if there is a model of \zfc without Efimov spaces, then in that model either $\bdd < \split = \continuum$, or else $\bdd < \continuum \geq \split \geq \aleph_\w$ and there are inner models containing large cardinals.

It would be interesting to determine whether the assumption on cofinalities in Theorem~\ref{thm:noseq} can be dropped:
\begin{question}
Is it consistent (relative to large cardinals) that $\sr < \noseq$?
\end{question}

\noindent Another question relevant to the hypothesis of Theorem~\ref{thm:noseq} is also an interesting combinatorial question in its own right:

\begin{question}
Is it consistent that $\mathrm{cof}(\sr) = \w$?
\end{question}

\noindent We have proved three lower bounds for $\sr$, namely $\split$, \non M, and $\bdd$. None of these three lower bounds can have countable cofinality. However, a different argument is required to show this for each of these three cardinals, and none of these three arguments readily generalizes to say anything about the cofinality of $\sr$. Note that it is easy to make $\sr$ singular: adding $\aleph_{\w_1}$ Cohen reals to a model of \ch, we obtain a model with $\cov M = \aleph_{\w_1} = \continuum$, which implies $\sr = \aleph_{\w_1}$.

While we have determined completely the place of $\sr$ among the cardinals in Cicho\'n's diagram, the same cannot be said for $\noseq$. A great deal toward this goal can be deduced from the theorems above (although we leave such deductions to the reader), but a few questions remain. In our view, the most interesting of these is:

\begin{question}\label{q:notLaver}
Is it consistent that $\non N < \noseq$?
\end{question}

\noindent If one is to avoid blowing the continuum up past $\aleph_\w$, then $\noseq \leq \sr$, in which case $\non N < \noseq$ implies $\non N < \sr \leq \max\{\bdd,\non N\} = \bdd$. Thus a positive answer to Question~\ref{q:notLaver} (if indeed the answer is positive) will most likely come from a model with $\non N < \bdd$.

The Laver model has $\non N < \bdd$, but will not work because, as mentioned in Section~\ref{sec:lower}, the second author has shown that $\noseq = \aleph_1$ in the Laver model. Another natural candidate is the model obtained by adding $\aleph_1$ random reals to a model of $\ma+\neg\ch$. 

\begin{question}
What is the value of $\noseq$ in a model obtained by adding $\aleph_1$ random reals to a model of $\ma+\neg \ch$?
\end{question}

\noindent By ``adding $\aleph_1$ random reals'' we mean forcing with the product measure algebra on $2^{\aleph_1}$. If instead one were to add the random reals in a length-$\w_1$ finite support iteration, then this forces 
$\noseq \leq \sr \leq \max\{\bdd,\non N\} = \aleph_1$ (with any ground model). In other words, adding the random reals iteratively adds a small Efimov space, even where none was present before. But it is not clear whether adding the random reals simultaneously will do so. Adding $\aleph_1$ random reals to a model of $\bdd = \aleph_1$ (either simultaneously or by iteration) will once again add small Efimov spaces, because it forces $\max\{\bdd,\non N\} = \aleph_1$. But it is not clear whether this happens when $\bdd > \aleph_1$ in the ground model.



\begin{thebibliography}{99}

\bibitem{B&J} T. Bartoszy\'nski and H. Judah, \emph{Set Theory: On theStructure of the Real Line}, A K Peters (1995).
\bibitem{Blass} A. Blass, ``Combinatorial cardinal characteristics of the continuum,'' in \emph{Handbook of Set Theory,} M. Foreman and A. Kanamori eds., Springer-Verlag (2010), pp. 395--489. 
\bibitem{rearrangement} A. Blass, J. Brendle, W. Brian, M. Hardy, J. D. Hamkins, and P. Larson, ``The rearrangement number,'' to appear in \emph{Transactions of the American Mathematical Society}. 
\bibitem{Blass&Shelah} A. Blass and S. Shelah, ``Ultrafilters with small generating sets,'' \emph{Israel Journal of Mathematics} \textbf{65} (1984), pp. 259--271.
\bibitem{Booth} D. Booth, ``A Boolean view of sequential compactness,'' \emph{Fundamenta Mathematicae} \textbf{85} (1974), pp. 99--102.
\bibitem{Brendle} J. Brendle, ``Around splitting and reaping," \emph{Comment. Math. Univ. Carolin.} \textbf{39} (1998), pp. 269--279. 
\bibitem{DoddJensen} A. J. Dodd and R. B. Jensen, ``The covering lemma for $\mathrm L[U]$,'' \emph{Annals of Mathematical Logic} \textbf{22} (1982), pp. 127--135.
\bibitem{vanDouwen} E. K. van Douwen, ``The integers and topology,'' in \emph{Handbook of Set-Theoretic Topology}, K. Kunen and J. E. Vaughan, eds., North-Holland, Amsterdam (1984), pp. 111--167.
\bibitem{vanDouwen&Fleissner} E. K. van Douwen and W. G. Fleissner, ``Definable forcing axiom: an alternative to Martin's axiom," \emph{Topology and its Applications} \textbf{35} (1990), pp. 277--289.
\bibitem{Dow} A. Dow, ``Efimov spaces and the splitting number,'' \emph{Topology Proceedings} \textbf{29} (2005), pp. 105--113.
\bibitem{Dow&Fremlin} A. Dow and D. Fremlin, ``Compact sets without converging sequences in the random real model,'' \emph{Acta Mathematica Universitatis Comenianae, New Series} \textbf{76} (2007), pp. 161--171.
\bibitem{Dow&Shelah} A. Dow and S. Shelah, ``An Efimov space from Martin's Axiom,'' \emph{Houston Journal of Mathematics} \textbf{39} (2013), pp. 1423--1435.
\bibitem{efimov} B. Efimov, ``The imbedding of the Stone-\v{C}ech compactifications of discrete spaces into bicompacta,'' \emph{Doklady Akademiia Nauk USSR} \textbf{189} (1969), pp. 244--246.
\bibitem{fedorcuk1} V. V. Fedor\v{c}uk, ``A bicompactum whose infinite closed subsets are all $n$-dimensional,'' \emph{Matematicheski\u{\i} Sbornik Novaya Seriya} \textbf{96}, no. 138 (1975), pp. 41--62.
\bibitem{fedorcuk2} V. V. Fedor\v{c}uk, ``Completely closed mappings, and the consistency of certain general topology theorems with the axioms of set theory,'' \emph{Matematicheski\u{\i} Sbornik Novaya Seriya} \textbf{99}, no. 141 (1976), pp. 1--26.
\bibitem{fedorcuk3} V. V. Fedor\v{c}uk, ``A compact space having the cardinality of the continuum with no convergent sequences,'' \emph{Mathematical Proceedings of the Cambridge Philosophical Society} \textbf{81} (1977), pp. 177--181.
\bibitem{Geschke} S. Geschke, ``The coinitiality of a compact space,'' \emph{Topology Proceedings} \textbf{30} (2006), pp. 237--250.
\bibitem{Gitik} M. Gitik, ``The negation of the singular cardinal hypothesis from $o(\kp) = \kp^{++}$,'' \emph{Annals of Pure and Applied Logic} \textbf{43} (1989), pp. 209--234.
\bibitem{hart} K. P. Hart, ``Efimov's problem," \emph{Open Problems in Topology II} (ed. Eliott Pearl), Elsevier Science, 2007, pp. 171--177.
\bibitem{Hausdorff} F. Hausdorff, ``Summen von $\aleph_1$ Mengen," \emph{Fundamenta Mathematicae} \textbf{26} (1936), pp. 241--255.
\bibitem{JMPS} W. Just, A. R. D. Mathias, K. Prikry, and P. Simon, ``On the existence of large $p$-ideals,'' \emph{Journal of Symbolic Logic} \textbf{55} (1990), pp. 457--465.
\bibitem{K&W} A. Kamburelis and B. W\k{e}glorz, ``Splittings,'' \emph{Archive for Mathematical Logic} \textbf{135} (1996), pp. 263--277.
\bibitem{Kechris} A. S. Kechris, \textit{Classical Descriptive Set Theory}, Graduate Texts in Mathematics vol. 156, Springer-Verlag, 1995.
\bibitem{Koppelberg} S. Koppelberg, ``Boolean algebras as unions of chains of subalgebras,'' \emph{Algebra Universalis} \textbf{7} (1977), pp. 195--203.
\bibitem{Miller} A. W. Miller, ``Covering $2^\w$ with $\w_1$ disjoint closed sets,'' in \emph{The Kleene Symposium}, eds. J. Barwise, H. J. Keisler, and K. Kunen, North-Holland (1980), pp. 415--421.
\bibitem{Oxtoby} J. C. Oxtoby, \emph{Measure and Category, second edition}, Springer-Verlag (1980).
\bibitem{Sobota} D. Sobota, ``The dominating number $\dom$ and convergent sequences," MathOverflow question (2015) \texttt{http://mathoverflow.net/q/206345}.
\bibitem{Sobota2} D. Sobota, ``The Nikodym property and cardinal characteristics of the continuum," \emph{Annals of Pure and Applied Logic} \textbf{170} (2019) pp. 1--35.


\end{thebibliography}
\end{document}